\documentclass[11pt,reqno]{amsart}
\usepackage{amsfonts,amstext}
\usepackage{amsmath,amstext,amssymb,amscd,mathrsfs,amsthm}
\usepackage{enumerate}
\usepackage[dvipsnames,usenames]{xcolor}

\newtheorem{thm}{Theorem}[section]
\newtheorem{lem}[thm]{Lemma}
\newtheorem{cor}[thm]{Corollary}
\newtheorem{prop}[thm]{Proposition}

\theoremstyle{definition}
\newtheorem{defn}[thm]{Definition}

\theoremstyle{remark}
\newtheorem{rmk}[thm]{Remark}
\numberwithin{equation}{section}

\usepackage[allbordercolors={1 1 1}]{hyperref}

\usepackage[disable]{todonotes}

\renewcommand{\H}[1][1]{W^{#1,2}(\Omega)}

\renewcommand{\to}{\mathrel{\rightarrow}}

\newcommand{\R}{\mathbb{R}}

\def\g{\gamma}

\def\G{\Gamma}
\def\O{\Omega}

\newcommand{\norm}[1]{\left\|#1\right\|}

\def\<{\langle}
\def\>{\rangle}

\def\H{H^1_{\G_0}}

\begin{document}

\author[B. Feng]{Baowei Feng}
\address{School of Mathematics,
Southwestern University of Finance and Economics,
Chengdu 611130, Sichuan, P. R. China }
 \email{bwfeng@swufe.edu.cn}
 
 \author[Y. Guo]{Yanqiu Guo}
\address{Department of Mathematics and Statistics, Florida International University, Miami FL 33199, USA }
 \email{yanguo@fiu.edu}

\author[M. A. Rammaha]{Mohammad A. Rammaha}
\address{Department of Mathematics, University of Nebraska--Lincoln, Lincoln, NE  68588-0130, USA} \email{mrammaha1@unl.edu}

\title[On a structural acoustic model]{On a structural acoustic model with logarithmic supercritical source terms}

\date{May 28, 2026}

\maketitle

\begin{abstract}
In this paper, we study a structural acoustic model consisting of a 
semilinear wave equation defined on a bounded domain 
\(\Omega\subset\mathbb{R}^3\), coupled with a Kirchhoff--Love plate 
equation acting on a flat portion of the boundary of \(\Omega\). 
Primarily for mathematical interest, we impose nonlinear damping terms and 
\emph{logarithmic-type} supercritical source terms on the system. 
We investigate local and global well-posedness, energy decay rates 
of potential well solutions, and blow-up of solutions under 
different conditions on the parameters and initial data. 
The main novelties include the analysis of the interaction between 
nonlinear damping and logarithmic energy-amplifying source terms, 
as well as the development of techniques for handling logarithmic source terms within 
potential well theory. The logarithmic nonlinearities are not homogeneous under scaling, which creates difficulties in studying potential well solutions. 
The wave-plate coupling through the acoustic pressure also causes technical difficulties in the analysis, especially in the proof of blow-up of weak solutions.
\end{abstract}

\begin{quote}
    {\footnotesize  \tableofcontents}
\end{quote}

\section{Introduction} \label{Intro}
In this paper, we study a structural acoustic model with nonlinear damping and \emph{logarithmic-type} source terms:
\begin{align}\label{PDE}
\begin{cases}
u_{tt}-\Delta u + |u_t|^{m-1}u_t = |u|^{p-1} u \log |u| &\text{ in } \O \times (0,T),\\[1mm]
w_{tt}+\Delta^2w + |w_t|^{r-1}w_t + u_t|_{\G} = |w|^{p-1} w \log |w| &\text{ in }\G\times(0,T),\\[1mm]
u=0&\text{ on }\G_0\times(0,T),\\[1mm]
\partial_\nu u=w_t&\text{ on }\G\times(0,T),\\[1mm]
w=\partial_{\nu_\G}w=0&\text{ on }\partial\G\times(0,T),\\[1mm]
(u(0),u_t(0))=(u_0,u_1),\hspace{5mm}(w(0),w_t(0))=(w_0,w_1).
\end{cases}
\end{align}
Here,  $\O\subset\R^3$ is a bounded, open, connected domain with a smooth boundary
$\partial\O=\overline{\G_0\cup\G}$, where $\G_0$ and $\G$ are two disjoint, open, connected sets of positive Lebesgue  measure.  Moreover, $\G$ is a \emph{flat} portion of the boundary of $\O$ and is referred to as the elastic wall. The part $\G_0$ of the boundary $\partial\O$ describes a rigid  wall, while  the coupling takes place on the flexible wall  $\G$. We also assume that the curve $\partial \Gamma = \partial \Gamma_0$ is smooth. The vectors $\nu$ and $\nu_\G$ denote the outer normals to $\G$ and $\partial\G$, respectively.

Models such as \eqref{PDE} arise in the context of modeling sound propagation in an acoustic chamber surrounded by a combination of rigid and flexible walls. The velocity potential $u$ of the acoustic field in the chamber is described by the solution to a wave equation, while vibrations of the flexible wall are described by the solution $w$ to a coupled Kirchhoff–Love plate equation. The kinematic compatibility condition $\partial_\nu u=w_t$ means that the fluid and the plate must move together at the interface $\Gamma$, i.e., the normal velocity of the fluid at the boundary equals the velocity of the flexible wall. Moreover, we remark that $w=\partial_{\nu_\G}w=0$ on the edge $\partial\G$ is the boundary condition of a clamped plate.

Mostly for mathematical interest, we impose the nonlinear damping terms \( |u_t|^{m-1}u_t \) and \( |w_t|^{r-1}w_t \), 
as well as the \emph{logarithmic-type} source terms \( |u|^{p-1} u \log |u| \) and \( |w|^{p-1} w \log |w| \). 
One difficulty in analyzing models with logarithmic nonlinear terms is that \( |u|^{p-1} u \log |u| \) changes sign depending on the value of \( |u| \). 
Another difficulty is that \( |u|^{p-1} u \log |u| \) is not homogeneous under scaling. Furthermore, we emphasize that the nonlinear source terms are energy-amplifying and may cause blow-up of solutions.

Throughout this paper, we assume \(1<p<6\) and define the function \( |s|^{p-1}s\log |s| \) to be \(0\) when \(s=0\), so that \( |s|^{p-1}s\log |s| \in C^1(\mathbb R) \). 
The parameters \(m\), \(r\), and \(p\) satisfy
\begin{align} \label{p}
m, r \geq 1, \quad 1 < p < 6, \quad p \frac{m+1}{m} < 6.
\end{align}
We remark that the condition \(p \frac{m+1}{m} < 6\) means that, as the power \(p\) of the source term approaches \(6\), the damping exponent \(m\) must be sufficiently large in order to control the source term in the well-posedness result. This study includes the \emph{supercritical} case ($3<p<6$), in which the source term $|u|^{p-1} u \log |u|$ is not locally Lipschitz from $H^1(\Omega)$ to $L^2(\Omega)$.

In system \eqref{PDE}, the coupling of the wave equation for $u$ and the plate equation for $w$ 
is realized through the boundary condition $\partial_\nu u = w_t$ on 
$\Gamma$ and the term $u_t|_{\Gamma}$ in the plate equation. Physically, 
$u_t|_{\Gamma}$ is the acoustic pressure load on the flexible wall 
$\Gamma$. Note that weak solutions $(u,u_t)$ of the wave equation belong 
to the finite energy space $H^1_{\Gamma_0}(\Omega) \times L^2(\Omega)$. 
Thus, $u_t$ belongs to $L^2(\Omega)$, and in general, without additional 
regularity, one cannot take the trace of an $L^2(\Omega)$ function on 
$\Gamma$. Therefore, in the definition of weak solutions below, namely, 
Definition~\ref{def:weaksln}, we test $u_t|_{\Gamma}$ with a test 
function $\psi$, and pass the time derivative to $\psi$ using integration 
by parts, and obtain the terms
\begin{align*}
\int_{\Gamma} u|_{\Gamma}(t) \psi_t(t)\,d\Gamma 
- \int_{\Gamma} u|_{\Gamma}(0) \psi_t(0)\,d\Gamma
- \int_0^t \int_{\Gamma} u|_{\Gamma} \psi_t \,d\Gamma \,d\tau,
\end{align*}
all of which make perfect sense since $u \in H^1_{\Gamma_0}(\Omega)$ is 
sufficiently smooth to take the boundary trace on $\Gamma$. Such an idea 
of integration by parts in time is used at various places in our analysis 
to handle the important term $u_t|_{\Gamma}$. Nevertheless, the lack of 
sufficient regularity of $u_t|_{\Gamma}$ still causes difficulties 
throughout the paper, in particular in proving the blow-up of weak 
solutions.

Structural acoustic models have a rich history.
For example, Beale \cite{Beale76} studied the spectral properties of an acoustic
boundary condition for the wave equation arising in theoretical acoustics.
Banks et al. \cite{Banks1991} introduced a two-dimensional model describing
acoustic-structure interaction in a rectangular acoustic cavity, where the boundary
consists of hard walls on three sides and a vibrating wall on the fourth side,
modeled by an Euler--Bernoulli beam equation.
Avalos and Lasiecka \cite{Avalos1,Avalos3,Avalos4} studied the stabilization,
controllability, and long-time behavior of structural acoustic models.
Related stability results for structural acoustic models with flexible curved walls
were obtained by Cagnol et al. \cite{Cagnol1}. Further developments involving more general structural components, such as
Reissner--Mindlin plates, Timoshenko beams, shear effects, and thermal effects,
can be found in the works of Grobbelaar-Van Dalsen \cite{MG1,MG3}.
We also mention the recent work of Becklin and Rammaha \cite{Becklin-Rammaha2}
on Hadamard well-posedness for a structural acoustic model with a supercritical
source and damping terms, as well as our previous works
\cite{FENG-JDE,FENG-JMAA} on asymptotic behavior and blow-up for related
structural acoustic models.
Furthermore, we refer the reader to the book \cite{Las2002} by Lasiecka and the
book \cite{VKEE} by Chueshov and Lasiecka, which provide comprehensive
overviews and references to many works on this topic. 
Other related contributions worthy of mention include \cite{Avalos2, Pelin, LAS1999,Vicente-jde}.

In system \eqref{PDE}, the logarithmic source terms compete with the damping terms: source terms amplify the energy while damping terms dissipate it. 
A classical result on wave equations with competing
damping term $|u_t|^{m-1}u_t$ and source term $|u|^{p-1}u$ was 
established by Georgiev and Todorova \cite{GT}. In particular, on a bounded 3D domain with $1<p\leq 3$,   
if the damping term dominates the source term ($m\geq p$), then weak 
solutions are global in time; whereas if the source term surpasses 
the damping term ($p>m$), then weak solutions may blow up if the initial 
energy is large enough. We refer to $p=3$ as the \emph{critical} exponent for the source term $|u|^{p-1}u$ in 3D because it is locally Lipschitz from $H^1(\Omega)$ to $L^2(\Omega)$ when $1<p\leq 3$.
An extension to a \emph{supercritical} source term 
($3< p < 6$ in 3D) was obtained by Bociu and Lasiecka 
\cite{BL3, BL2, BL1}. Furthermore, for a nonlinear wave equation 
with a rapidly growing polynomial source term (including the 
range $p\geq 6$ in a 3D periodic domain), Guo \cite{Guo98} established the global well-posedness of weak solutions, provided the damping term 
has sufficiently fast growth rate 
($m\geq \frac{3}{2}p - \frac{5}{2}$ if $p\geq 6$) and the initial 
data have higher integrability. There are many other interesting works 
concerned with the competition of source terms and various types of 
damping terms in nonlinear wave equations. See, for example, 
\cite{BLR1, GR, GRSTT, MOHNICK2} and references therein. 
A main novelty of this work is the introduction of logarithmic source terms into structural acoustic models. In the 
literature, logarithmic-type nonlinearities have been introduced and 
studied for nonlinear wave equations. For example, the work of Tao 
\cite{Tao2007} established global regularity for a logarithmically 
supercritical defocusing nonlinear wave equation in 3D with a 
defocusing nonlinearity $u^5 \log(2+u^2)$ for spherically symmetric 
initial data. For physical motivations for logarithmic nonlinearity, we refer the reader to \cite{BialynickiBirula1975, Hefter1985}. For other 
works on the analysis of wave equations with logarithmic nonlinearity, see, e.g., \cite{  DiShangSong2020,  HaPark2021,   Kafini2020}.

The main focus of this manuscript is the potential well solutions of 
system \eqref{PDE}. The study of potential well solutions for nonlinear hyperbolic equations has a long history, going back to the foundational works of Sattinger \cite{Sattiinger68} and Payne–Sattinger \cite{PS}. In particular, Payne and 
Sattinger \cite{PS} considered a nonlinear hyperbolic equation in the 
canonical form:
\begin{align}\label{canonical}
u_{tt} = \Delta u + f(u), \;\; \text{with} \;\; u(0)=u_0, 
\;\; u_t(0)=u_1,
\end{align}
where $u=0$ on the boundary of a smooth bounded domain 
$\Omega\subset \mathbb{R}^n$. An important quantity for equation 
\eqref{canonical} is the potential energy functional 
$J(u)=\frac{1}{2}\|\nabla u\|_{L^2(\Omega)}^2 - \int_{\Omega} F(u)\,dx$, 
where $F(u)= \int_0^u f(s)\,ds$. The depth of the potential well $d$ is defined as the mountain pass level, i.e., 
$d=\inf_{u \neq 0} \sup_{\lambda>0} J(\lambda u)$. Assume that the initial total energy 
$E(0) = \frac{1}{2}\|u_1\|_{L^2(\Omega)}^2 + J(u_0)$ satisfies $E(0) < d$. If $\|\nabla u_0\|_{L^2}^2 > \int_{\Omega} u_0 f(u_0)\,dx$, then the weak solution is global in time; whereas if 
$\|\nabla u_0\|_{L^2}^2 < \int_{\Omega} u_0 f(u_0)\,dx$, then the solution blows 
up in finite time. In the same spirit as \cite{PS, Sattiinger68}, we establish the 
global existence and blow-up of potential well solutions for the 
structural acoustic model \eqref{PDE}. We remark that the literature on potential well solutions for nonlinear wave equations typically focuses on power-type source terms, which are homogeneous under scaling.
However, our system includes logarithmic source terms that are not homogeneous under scaling. This non-homogeneity causes difficulties in employing the potential well theory and constitutes a novel aspect of this work.

The acoustic pressure term $u_t|_{\Gamma}$ is the key coupling term connecting the wave equation and the plate equation in system \eqref{PDE}, 
but the lack of sufficient regularity of this term causes technical difficulties in the proof of the blow-up theorem. In our previous paper 
\cite{FENG-JMAA} on a structural acoustic model with damping and power-type source terms, we replaced the d'Alembertian 
$u_{tt} - \Delta u$ with the Klein-Gordon operator 
$u_{tt} - \Delta u + u$ to overcome the technical difficulty in the proof of the blow-up result. In this paper, we improve our argument and establish the blow-up of weak solutions without altering the d'Alembertian. 
In particular, the nonlinear damping term $|w_t|^{r-1}w_t$ plays a crucial role in treating the acoustic pressure term $u_t|_{\Gamma}$ in the proof of the blow-up results.

We emphasize that the stabilization estimate (\ref{4-12}) in Lemma \ref{lem4-1} does not contain lower-order terms, such as $\|u\|_2^2 + |w|_2^2$, on the right-hand side of the inequality. This feature greatly shortens the proof, as it completely avoids the traditional compactness-uniqueness argument (see, e.g., \cite{GR2}) that is otherwise needed to absorb lower-order terms. This type of refined estimate was originally introduced in \cite{GRS3} by Guo et al.\ for the same purpose of eliminating lower-order terms.

The paper is organized as follows. Section \ref{sec-Main} provides statements of all main results. Section \ref{sec-Local} establishes the local well-posedness of weak solutions, as well as global existence when the damping terms are stronger than the source terms. In Section \ref{sec-Well}, we prove the global existence of potential well solutions. 
In Section \ref{sec-decay}, we establish exponential and polynomial energy decay rates, depending on the powers of the damping terms. 
In Section \ref{sec-Blow}, we prove finite-time blow-up of weak solutions when the source terms are stronger than the damping terms, in the cases of both negative and positive initial energy.

\vspace{0.1 in}

\section{Main Results} \label{sec-Main}

This section is devoted to stating the main results of this paper. Before presenting these results, we first introduce some notation used throughout the paper.

\subsection{Notation}

Throughout the paper, the following notation for $L^p$ space 
norms and inner products is used:
\begin{align*}
&\|u\|_p = \|u\|_{L^p(\Omega)}, 
&&(u,v)_\Omega = (u,v)_{L^2(\Omega)}, \\
&|u|_p = \|u\|_{L^p(\Gamma)}, 
&&(u,v)_\Gamma = (u,v)_{L^2(\Gamma)}.
\end{align*}
Note that double bars $\|\cdot\|_p$ denote norms on $\Omega$, 
while single bars $|\cdot|_p$ denote norms on $\Gamma$.
We write $\gamma u$ (or equivalently $u|_{\Gamma}$) to denote 
the trace of $u$ on $\Gamma$.
Throughout, $C > 0$ denotes a generic positive constant that 
may differ from line to line.

In addition, we define
\begin{align*} 
\H(\O):= \{ u \in H^1(\Omega) : u|_{\Gamma_0} = 0 \}.
\end{align*}

By the Poincar\'{e} inequality, the norm  $\|\nabla u\|_2$ is equivalent to the standard $H^1$ norm on $\H(\O)$.  
We therefore equip $\H(\O)$ with the norm and 
inner product:
\begin{align*} 
\|u\|_{\H(\O)} = \|\nabla u\|_2, 
\qquad 
(u,v)_{\H(\O)} = (\nabla u, \nabla v)_\Omega.
\end{align*}
Analogously, by the Poincar\'{e} inequality on $\Gamma$, 
we equip $H^2_0(\Gamma)$ with the norm and inner product:
\begin{align*} 
\|w\|_{H^2_0(\Gamma)} = |\Delta w|_2, 
\qquad 
(w,z)_{H^2_0(\Gamma)} = (\Delta w, \Delta z)_\Gamma.
\end{align*}
Let $Y$ be a Banach space. We denote the duality pairing 
between the dual space $Y'$ and $Y$ by 
$\langle \psi, y \rangle_{Y',Y}$, or simply by 
$\langle \cdot, \cdot \rangle$. That is,
\begin{align*}
\langle \psi, y \rangle = \psi(y) 
\quad \text{for } y \in Y,\, \psi \in Y'.
\end{align*}

\vspace{0.1 in}

\subsection{Local and Global Well-Posedness of Weak Solutions}
In this subsection, we present results on the local well-posedness of weak solutions, as well as the global existence of weak solutions when the damping terms are stronger than the source terms.

To state these results, we first introduce the notion of a weak solution for system \eqref{PDE}, assuming the parameters \(m\), \(r\), and \(p\) satisfy \eqref{p}.

\begin{defn}\label{def:weaksln}
We call $(u,w)$   a \emph{weak solution} of \eqref{PDE} on the interval $[0,T]$ if
    \begin{enumerate}[(i)]
        \setlength{\itemsep}{5pt}
        \item $u\in C([0,T];H^1_{\G_0}(\O))$, $u_t\in C([0,T];L^2(\O))\cap L^{m+1}(\O\times(0,T))$,
        \item $w\in C([0,T];H^2_0(\G))$, $w_t\in C([0,T];L^2(\G))\cap L^{r+1}(\G\times(0,T))$,
        \item $(u(0),u_t(0))=(u_0,u_1) \in H^1_{\G_0}(\O)\times L^2(\O)$,
        \item $(w(0),w_t(0))=(w_0,w_1) \in H^2_0(\G)\times L^2(\G)$,
        \item The functions $u$ and $w$ satisfy the following weak formulation for all $t\in[0,T]$:
\begin{align}\label{wkslnwave}
(u_{t}(t),  \phi(t))_\O & - (u_1,\phi(0))_\O-\int_0^t ( u_t, \phi_t)_\O \, ds
+\int_0^t (\nabla u, \nabla\phi )_\O \, ds \notag \\
&-\int_0^t  (w_t, \g \phi )_\G \, ds
+ \int_0^t \int_{\Omega}  |u_t|^{m-1} u_t \phi \, dx ds  \notag\\
&= \int_0^t\int_\Omega (|u|^{p-1} u \log |u|) \phi \, dx ds,
\end{align}
\begin{align}\label{wkslnplt}
&(w_t(t)  + \g u(t),\psi(t) )_\G  -(w_1 +\g u_0 ,\psi(0))_\G -\int_0^t (w_t, \psi_t )_\G \, ds \notag \\
&\quad -\int_0^t (\g u, \psi_t)_\G \, ds + \int_0^t (\Delta w, \Delta\psi)_\G \, ds \notag \\
&\quad+ \int_0^t \int_{\Gamma}  |w_t|^{r-1} w_t \psi \, d\Gamma  ds = \int_0^t \int_{\G} (|w|^{p-1} w\log |w|) \psi \, d\G ds,
\end{align}
for all test functions $\phi$ and $\psi$ satisfying
$\phi\in C([0,T];H^1_{\G_0}(\O))  \cap L^{m+1}(\O\times(0,T))$, $\psi\in C\left([0,T];H^2_0(\G)\right)$  with
$\phi_t\in L^1(0,T;L^2(\O))$, and $\psi_t\in L^{1}(0,T;L^2(\G))$.
    \end{enumerate}
\end{defn}

\vspace{0.1 in}

Next, we define the \emph{quadratic energy} $E(t)$ by 
\begin{align}\label{3-9-1}
E(t)=\frac{1}{2}\left(\|u_t(t)\|_2^2+|w_t(t)|_2^2   +\|\nabla u(t)\|_2^2   +|\Delta w(t)|_2^2\right).
\end{align}

The following result establishes the local well-posedness of weak solutions.
\begin{thm} [{\bf Local well-posedness of weak solutions}]  \label{t:1}
 Assume the parameters \( m \), \( r \) and \( p \) satisfy \eqref{p}. Then the following statements hold:
\begin{enumerate}[(i)]
\item  Assume the initial data
$(u_0, w_0, u_1, w_1) \in \mathcal{H}$ where the space $\mathcal H$ is defined as 
$$\mathcal{H} = H^1_{\Gamma_0}(\Omega) \times H^2_0(\Gamma)  \times L^2(\Omega) \times L^2(\Gamma).$$
Then system  \eqref{PDE} has a local weak solution $(u,w)$ defined on $[0,T_0]$ for some $T_0>0$ depending on the initial quadratic energy $E(0)$. In addition, the following energy identity holds for all $t\in[0,T_0]$:
	\begin{align}\label{energy}
	&E(t) + \int_0^t\int_{\Omega} |u_t|^{m+1} dx ds +  \int_0^t\int_\G  |w_t|^{r+1} d\G ds	 \notag\\
	&=E(0)+     \int_0^t\int_{\Omega}   (|u|^{p-1} u \log |u|)     u_t \, dx ds            
    +       \int_0^t\int_\G     (|w|^{p-1} w \log |w|)    w_t      \,   d\G ds.
	\end{align}
\item Assume the initial data $U_0= (u_0, w_0, u_1, w_1) \in Y$ where the space $Y$ is defined as
$$ Y  =\left(H^1_{\Gamma_0}(\Omega) \cap L^{\frac{3(p-1)}{2}}(\Omega)\right) 
    \times H^2_0(\Gamma) \times L^2(\Omega) \times L^2(\Gamma).$$
Then weak solutions are unique. 
\item If $U_0^n = (u_0^n, w_0^n, u_1^n, w_1^n)$ is a sequence of initial data such 
that, as $n \longrightarrow \infty$,
\begin{equation*}
    U_0^n \longrightarrow U_0 \quad \text{in } Y,
\end{equation*}
then the corresponding weak solutions $(u^n, w^n)$ and $(u, w)$ of \eqref{PDE} 
satisfy:
\begin{equation*}
    \left(u^n, w^n, u^n_t, w^n_t\right) \longrightarrow (u, w, u_t, w_t) 
    \quad \text{in } C([0,T];\, \mathcal{H}), 
    \quad \text{as } n \longrightarrow \infty.
\end{equation*}
\end{enumerate}
  \end{thm}

We remark that the set $H^1_{\Gamma_0}(\Omega) \cap L^{\frac{3(p-1)}{2}}(\Omega)$ coincides with $H^1_{\Gamma_0}(\Omega)$
if $1<p\leq 5$ because of the embedding $H^1_{\Gamma_0}(\Omega) \hookrightarrow L^6(\Omega)$.
If $p>5$, then $H^1_{\Gamma_0}(\Omega) \cap L^{\frac{3(p-1)}{2}}(\Omega)$ is a proper subset of $H^1_{\Gamma_0}(\Omega)$.

The next result concerns the global existence of weak solutions. More precisely, if the damping terms are stronger than the source terms, then weak solutions exist globally in time.  
\begin{thm} [{\bf Global weak solutions}]  \label{t:2}
Assume the parameters \( m \), \( r \) and \( p \) satisfy \eqref{p}. Let $(u_0, w_0, u_1, w_1) \in \mathcal{H}$ with $u_0 \in L^{p+1+\sigma}(\Omega)$, where $\sigma>0$. If $p< \min\{m,r\}$, then
the solution $(u,w)$ obtained in Theorem \ref{t:1} is  global in time, and $T_0$ can be taken arbitrarily large.
\end{thm}

We remark that the condition $u_0 \in L^{p+1+\sigma}(\Omega)$, where $\sigma>0$, means that $u_0$ has slightly stronger integrability than an $L^{p+1}(\Omega)$ function requires.

\vspace{0.1 in}

\subsection{Potential Well Solutions: Global Existence and Energy Decay Rates}

In this subsection, we present results on the global existence of potential well solutions and their energy decay rates. Before stating these results, we introduce the concepts of the potential well and the Nehari manifold.

Let $X=H^1_{\Gamma_0}(\Omega)\times H^2_0(\Gamma)$, endowed with the norm:  
\begin{align}    \label{normX}
\norm{(u,w)}_X  =   (\norm{\nabla u}_2^2 +  |\Delta w|_2^2)^{1/2}.
\end{align}
Throughout this subsection, we assume $1<p<5$. The functional $J:X\to \mathbb{R}$ is defined by
\begin{align}\label{3-8}
J(u,w)&:=\frac{1}{2}(\|\nabla u\|^2_2 + |\Delta w|^2_2) - \frac{1}{p+1} \left(\int_\Omega |u|^{p+1}\log|u|dx + \int_\Gamma |w|^{p+1}\log |w|d\Gamma\right)\notag\\
&\quad+ \frac{1}{(p+1)^2} \left(\|u\|^{p+1}_{p+1} +  |w|^{p+1}_{p+1} \right).
\end{align}
By the Sobolev embeddings $H^1_{\Gamma_0}(\Omega) \hookrightarrow L^6(\Omega)$ and $H^2_0(\Gamma) \hookrightarrow L^{\infty}(\Gamma)$, 
together with the assumption $1<p<5$, the functional $J(u,w)$ is well-defined and $J: X \rightarrow \mathbb R$ is continuous.

The \emph{potential energy} of the system is given by $ J(u(t),w(t))$. 
The Fr\'echet derivative of the functional $J$ at $(u,w) \in X$ is given by 
\begin{align} \label{J'}
\langle J'(u,w), (\phi, \psi) \rangle = & 
\int_{\Omega} \nabla u \cdot \nabla \phi \, dx + \int_{\Gamma} \Delta w \cdot \Delta \psi \, d\Gamma  \notag\\
& - \int_{\Omega} (|u|^{p-1} u \log |u|) \phi \, dx - \int_{\Gamma} (|w|^{p-1} w \log |w|) \psi \, d\Gamma,
\end{align}
for $(\phi,\psi) \in X$. 

The \emph{Nehari manifold} $\mathcal N$ associated with the functional $J$ is defined as  
\begin{align} \label{neha} 
\mathcal N := \{ (u,w) \in X \backslash \{(0,0)\}:  \langle  J'(u,w), (u,w)   \rangle =0 \}.
\end{align}
We refer the reader to the original works of Nehari \cite{nehari1960class, nehari1961characteristic}. The Nehari functional $I$ is defined by
\begin{align}\label{3-9-0}
I(u,w)=\|\nabla u\|^2_2+|\Delta w|^2_2-\int_\Omega |u|^{p+1} \log |u| dx-\int_\Gamma |w|^{p+1} \log |w| d\Gamma.
\end{align}
It follows from \eqref{J'} and \eqref{3-9-0} that the Nehari manifold \eqref{neha} can be written as
\begin{align}   \label{Nehari}
\mathcal{N}=\left\{(u,w)\in X\backslash \{(0,0)\}: I(u,w)=0 \right\}.
\end{align}

We define the \emph{potential well} associated with the potential energy $J(u,w)$ by 
\begin{align}    \label{def-W}
W&:=\{(u,w)\in X: J(u,w)<d\}.
\end{align}
Here, $d$ denotes the \emph{depth of the potential well} $W$, defined by
\begin{align}    \label{depth}
d:=\inf_{(u,w)\in\mathcal{N}}J(u,w).
\end{align} 

The following proposition shows that $d$ is strictly positive.
\begin{prop}\label{lem3-1}
Let $1<p<5$. Then the constant $d$ defined in \eqref{depth} is strictly positive.
\end{prop}

Since $J(0,0) =0$ and $d>0$, we have $(0,0) \in W$. Furthermore, since $J$ is continuous, a neighborhood of $(0,0)$ in $X$ is contained in $W$. 
It follows from (\ref{def-W}) and (\ref{depth}) that the potential well $W$ and the Nehari manifold $\mathcal N$ are disjoint, i.e.,
\begin{align}   \label{disjointWN}
W \cap \mathcal N =   \emptyset.
\end{align}

We decompose the potential well $W$ into two regions:  
\begin{align}
W_1&=\left\{(u,w)\in W:I(u,w)>0\right\} \cup\{(0,0)\}, \label{defW1}\\
W_2&=\left\{(u,w)\in W:I(u,w)<0\right\}. \label{defW2}
\end{align}
Clearly, $W_1\cap W_2=\emptyset$. 
In addition, by \eqref{disjointWN}, it follows that 
$$W_1\cup W_2=W.$$
We will show that if the initial data $(u_0,w_0)$ belong to $W_1$ and the initial total energy is less than the depth $d$ of the potential well, 
then the weak solution to system \eqref{PDE} is global in time. 
Moreover, under additional assumptions, the uniform decay rates of energy will be established. 
Thus, we call the set $W_1$ the ``stable" region of $W$. 
However, the solution may blow up in finite time if the initial data belong to $W_2$, as shown in Corollary \ref{cor1}. Hence we call $W_2$ the ``unstable" region.

\begin{prop} \label{prop-W1}
The set $W_1$ defined in (\ref{defW1}) is nontrivial. More precisely, there exists $r>0$ such that 
the ball 
$B_{r}(0,0) = \{(u,w) \in X: \|(u,w)\|_X < r\}$ is contained in $W_1$.  
\end{prop}

We define the \emph{total energy} $\widehat E(t)$ as the sum of the kinetic energy and the potential energy:
\begin{align}\label{3-9}
\widehat{E}(t): &=   \frac{1}{2}(\|u_t(t)\|^2_2+|w_t(t)|^2_2)+J(u(t),w(t)),
\end{align}
where the functional $J$ was defined in (\ref{3-8}).

The following result shows that if $(u_0,w_0)\in W_1$ and $\widehat{E}(0)<d$, then the weak solution can be extended to a global weak solution.

\begin{thm} [{\bf Global existence of potential well solutions}]
\label{thm3-1}
Assume the parameters \( m \), \( r \) and \( p \) satisfy \eqref{p} and $1< p < 5$. Assume further that
$(u_0,w_0)\in W_1$ and the total energy $\widehat{E}(0)<d$. Then system \eqref{PDE} has a unique global weak solution $(u,w)$. In addition,  for
any $t\geq 0$, we have
\begin{enumerate}[(i)]
\item  $J(u(t),w(t))\leq \widehat{E}(t)\leq\widehat{E}(0)<d$,
\smallskip
\item  $(u(t),w(t))\in W_1$,
\smallskip
\item  $0\leq E(t)\leq \frac{p+1}{p-1} \widehat E(t) < \frac{p+1}{p-1} d$.
\end{enumerate}
\end{thm}

\smallskip

\begin{thm} [{\bf Energy decay rates}]\label{thm4-1}
Assume the parameters \( m \), \( r \) and \( p \) satisfy \eqref{p}, with $1\leq m\leq 5$.
Fix an arbitrary $d_0 \in (0,d)$. Assume further that $(u_0,w_0)\in W_1$ and $\widehat{E}(0) \leq d_0 <d$. 
Then the global weak solution of problem \eqref{PDE} established in Theorem \ref{thm3-1} satisfies the following decay rates.
\begin{enumerate}[(i)]
\item If $m=r=1$, then
the total energy $\widehat{E}(t)$  decays to zero exponentially:
\begin{align}\label{exp}
\widehat{E}(t)\leq  C\widehat{E}(0) e^{-at},\ \ \forall\ t\geq0,
\end{align}
where $C$ and $a$ are positive constants independent of the initial data.

\smallskip

\item If at least one of $m$ and $r$ is not equal to 1, then the total energy $\widehat{E}(t)$  decays algebraically:
\begin{align}\label{alg}
\widehat{E}(t)\leq  C(\widehat E(0)) (1+t)^{-\frac{2}{\eta-1}},     \ \ \forall\ t\geq0,    
\end{align}
where $\eta=\max\{m,r\}$ and $C(\widehat E(0))$ is a positive constant depending on $\widehat E(0)$.
\end{enumerate}
\end{thm}

\begin{rmk}
By \eqref{p} and \(1\leq m\leq 5\), we have \(1< p < 5\) in Theorem \ref{thm4-1}. 
Moreover, according to Theorem \ref{thm4-1}, a larger value of \(m\) leads to a slower decay rate as \(t \to \infty\). 
The fastest exponential decay rate occurs when \(m = r = 1\), namely, when the damping terms are linear. 
\end{rmk}

\vspace{0.1 in}

\subsection{Finite-Time Blow-Up}
Finally, we present two blow-up results: blow-up with negative initial energy and blow-up with positive initial energy.

The first blow-up result shows that if the initial energy is negative and the source terms are stronger than the damping terms, then weak solutions of \eqref{PDE} blow up in finite time.

\begin{thm}[\bf Blow-up with negative initial energy]\label{thm6-1}
Assume the parameters \( m \), \( r \), and \( p \) satisfy \eqref{p}. Moreover, suppose $p>m$, $p>r>1$, and $\widehat{E}(0)<0$. Then the weak solution $(u(t),w(t))$ of system \eqref{PDE} blows up in finite time. In particular,
$$
\limsup_{t \to T^-} (\|\nabla u(t)\|_2^2+|\Delta w(t)|_2^2)= +\infty,
$$
for some $0<T<\infty$, where $T$ has an explicit upper bound, depending on the initial data, given in (\ref{Tbound}).
\end{thm}

\vspace{0.1 in}

Before stating the second blow-up result, we introduce several constants. We define the best embedding constants:
\begin{align}\label{3-12}
\alpha_1:=\sup_{u\in
H^1_{\Gamma_0}(\Omega)\backslash\{0\}}\frac{\|u\|^{6}_{6}}{\|\nabla
u\|^{6}_2},\ \ 
\alpha_2:=\sup_{w\in
H^2_0(\Gamma)\backslash\{0\}}\frac{|w|^{6}_{6}}{|\Delta
w|^{6}_2}.
\end{align}

The function $\Psi:\mathbb{R}^+\to\mathbb{R}$ is defined by
\begin{align}\label{b1}
\Psi(z):=z - \frac{8(\alpha_1 + \alpha_2)}{(p+1) e (5-p)} z^{3}.
\end{align}
The function $\Psi(z)$ is continuously differentiable, concave, and
has its maximum at $z = z_0 = \frac{1}{2} \left[ \frac{(p+1) e (5-p)}{6(\alpha_1 + \alpha_2) } \right]^{\frac{1}{2}} >0$. Define
\begin{align*}
\hat{d}:=\sup_{[0,\infty)}\Psi(z)=\Psi(z_0).
\end{align*}  

In addition, we define a constant $B>0$ satisfying
\begin{align}  \label{defB}
\frac{p-1}{4} z_0 = \frac{p+3}{2}B+ C_2 B^{\frac{r+1}{2(r-a)}},
\end{align}
where $a \in (0,\frac{1}{2})$ is given in \eqref{6-4} and $C_2 = \frac{r-a}{r}\left(\frac{4a}{r(p-1)}\right)^{\frac{a}{r-a}}  \left(\frac{r}{r+1} \right)^{\frac{r}{r-a}}$. 
From (\ref{defB}), the constant $B$ satisfies
\begin{align}\label{Bz}
B <  z_0.
\end{align}

Then we have the following result regarding the blow-up of solutions when the initial total energy is positive.

\begin{thm} [{\bf Blow-up with positive initial energy}]\label{thm6-2}
Assume the parameters \( m \), \( r \), and \( p \) satisfy \eqref{p}.
Further assume
\begin{align}\label{y0} 
E(0)>z_0,\ \ \mbox{and}\ 0\leq\widehat{E}(0)< \min \{B, \hat{d}\}.
\end{align}
Then the weak solution $(u(t),w(t))$ of \eqref{PDE} blows up in
finite time provided $p>m$ and $p>r>1$. In particular,
$$
\limsup_{t\rightarrow T^-}(\|\nabla u(t)\|_2^2+|\Delta w(t)|_2^2) =+\infty,
$$
for some $0<T<\infty$.
\end{thm}

\vspace{0.1 in}

\begin{rmk}
In Theorem \ref{thm6-2}, the two conditions \(E(0)>z_0\) and \(0\leq \widehat{E}(0)< \min \{B, \hat{d}\}\) are compatible, and initial data satisfying both can be constructed by a scaling argument. Indeed, consider nonzero initial data \((u_0, w_0) = \lambda (\phi, \psi)\). Then
\begin{align} \label{scaling1}
J(u_0,w_0) \sim - \lambda^{p+1} \log \lambda
\quad \text{as } \lambda \rightarrow \infty.
\end{align}
Thus, we can choose \(\lambda\) sufficiently large so that \(J(u_0,w_0)<0\). Next, we can choose \((u_1,w_1)\) such that
\begin{align} \label{scaling2}
\widehat E(0) = \frac{1}{2} \left(\|u_1\|_2^2 + |w_1|_2^2\right) + J(u_0,w_0)
\in \left[0, \min \{B, \hat{d}\}\right).
\end{align}
By (\ref{scaling1}) and (\ref{scaling2}), we have \(E(0) \sim \lambda^{p+1} \log \lambda\), and thus \(E(0) > z_0\) when \(\lambda\) is sufficiently large. Therefore, the set of initial data leading to blow-up in Theorem \ref{thm6-2} is nonempty and rich.
\end{rmk}

\vspace{0.1 in}

The following corollary concerns the blow-up of solutions when the initial data belong to the unstable region $W_2$.

\begin{cor}\label{cor1}
Assume that \eqref{p} holds, $p>m$, $p>r>1$, $0\leq\widehat{E}(0)< \min \{B, \hat{d}\}$, and $(u_0,w_0)\in W_2$. Then the weak solution $(u(t),w(t))$ of \eqref{PDE} blows up in finite time.
\end{cor}

\vspace{0.1 in}

\section{Local and Global Well-Posedness of Weak Solutions} \label{sec-Local}

This section addresses the local well-posedness of weak solutions to system \eqref{PDE}, as well as the global existence of weak solutions when the damping terms are stronger than the source terms.

\subsection{Local Well-Posedness}
Theorem \ref{t:1} establishes the local well-posedness of weak solutions to our model (\ref{PDE}). The local well-posedness includes the local existence, uniqueness, and continuous dependence on initial data. 
We note that Theorem \ref{t:1} is a corollary of Theorems 1.4, 1.7 and 1.9 in \cite{Becklin-Rammaha2} by Becklin and Rammaha. 
In particular, it suffices to verify that the logarithmic source terms in our model (\ref{PDE}) satisfy Assumptions 1.1 and 1.6 in \cite{Becklin-Rammaha2}. More precisely, \cite{Becklin-Rammaha2} establishes the local existence of weak solutions for a structural acoustic model, where the source term $f(u)$ is a general $C^1(\mathbb R)$ function satisfying 
\begin{align} \label{p1-1}
|f'(s)| \leq C (|s|^{p_1-1} +1), \;\;\text{with} \; 1 \leq p_1 <6 \;\;\text{and} \;\; p_1 \frac{m+1}{m}<6.
\end{align}
To obtain uniqueness and continuous dependence on initial data, \cite{Becklin-Rammaha2} additionally requires that, for $p_1\geq 3$, $f\in C^2(\mathbb R)$ satisfying
\begin{align} \label{p1-2}
|f''(s)| \leq C (|s|^{p_1-2} +1).
\end{align}

Recall that we define the function $|s|^{p-1} s \log |s|$ to be zero when $s=0$, where $p>1$. 
More precisely, we identify the function $|s|^{p-1} s \log|s|$ with the following $C^1$ function $f(s)$:  
\begin{align}  \label{def-f}
f(s) = 
\begin{cases}
|s|^{p-1} s \log |s|, & \text{if } s \neq 0, \\
0, & \text{if } s = 0,
\end{cases}
\end{align}
where $p>1$. Moreover, if $p>2$, then $f$ defined in (\ref{def-f}) belongs to $C^2(\mathbb R)$.

\begin{prop}  \label{prop-f}
Let $f$ be the function defined in (\ref{def-f}), and let $\sigma>0$. Then, for all $s\in \mathbb R$,
\begin{align}
&|f'(s)| \leq C ( |s|^{(p+\sigma)-1} + 1),     \;\; \text{if} \;\;  p>1; \label{propf-1} \\
&|f''(s)| \leq C ( |s|^{(p+\sigma) -2} + 1),     \;\; \text{if} \;\;  p>2. \label{propf-2}
\end{align}
\end{prop}

\begin{proof}
For any $\sigma>0$, we have
\begin{align}  \label{ineq-log}
|s^\sigma\log s|\leq \frac{1}{e\sigma} \ \mbox{for}\ 0<s<1,\ \ \ s^{-\sigma}\log s\leq  \frac{1}{e\sigma}   \ \mbox{for}\ s\geq 1.
\end{align}
We emphasize that the simple algebraic facts listed in (\ref{ineq-log}) are important for the analysis in this paper. 
Essentially, it shows that a power function $s^{\sigma}$, with any small positive power $\sigma$, can absorb the singularity of $\log s$ as $s\rightarrow 0^+$.  
Moreover, $s^{\sigma}$ grows faster than $\log s$ as $s\rightarrow \infty$, and thus $s^{-\sigma} \log s$ remains bounded as $s\rightarrow \infty$.

By differentiating (\ref{def-f}), we obtain
\begin{align*}
|f'(s)| &=  |s|^{p-1} +  p \left|\log |s| \right|  |s|^{p-1} \notag\\
& =  |s|^{p-1} +  p \left| |s|^{-\sigma}    \log |s| \right|  |s|^{p-1+\sigma} \notag\\
& \leq   |s|^{p-1} + p (e\sigma)^{-1}  |s|^{p-1 + \sigma}, \;\;\text{if}\;\; |s| \geq 1,
\end{align*}
where we have used (\ref{ineq-log}).

Moreover, since $p>1$, we obtain from (\ref{ineq-log}) that
\begin{align*}
|f'(s)| &=  |s|^{p-1} +  p \left|\log |s| \right|  |s|^{p-1}  \notag\\
& \leq   |s|^{p-1} +   p [e(p-1)]^{-1},  \;\;\text{if}\;\;  0<|s| < 1.
\end{align*}
Also, it is straightforward to verify that $f'(0) =0$.

By Young's inequality, for all $s\in \mathbb R$,
\begin{align*}
|f'(s)| \leq C (|s|^{p-1} + |s|^{p-1 + \sigma} + 1) \leq C ( |s|^{p-1 + \sigma} + 1), \;\; \text{where}  \;\; p>1, \; \sigma>0.
\end{align*}

Similarly, one can verify (\ref{propf-2}). 

\end{proof}

Since $1<p<6$, we have $1<p+\sigma<6$ for all sufficiently small $\sigma>0$. 
Moreover, since $p\frac{m+1}{m}<6$ by (\ref{p}), we have $(p + \sigma)\frac{m+1}{m}<6$ for all sufficiently small $\sigma>0$.  
Therefore, setting $p_1=p+\sigma$ and applying Proposition \ref{prop-f}, we conclude that the logarithmic source term $f(u) = |u|^{p-1} u \log |u|$ satisfies (\ref{p1-1}) and (\ref{p1-2}).
As a result, Theorem \ref{t:1} is a corollary of Theorems 1.4, 1.7 and 1.9 in \cite{Becklin-Rammaha2}.

\vspace{0.1 in}

\subsection{Global Existence with Dominating Damping Terms}

Theorem \ref{t:2} states that the local weak solution established in Theorem \ref{t:1} can be extended to a global solution if 
$u_0 \in L^{p+1+\sigma}(\Omega)$ with $\sigma>0$, and the damping terms are stronger than the source terms, namely, 
$\min\{m,r\}>p$. We note that Theorem \ref{t:2} is a corollary of Theorem 1.5 in \cite{Becklin-Rammaha2}. Recall that Theorem 1.5 in \cite{Becklin-Rammaha2} states that for the structural acoustic model with a general source term $f(u)$ satisfying (\ref{p1-1}), if 
\begin{align} \label{p1-3}
\min\{m,r\}\geq p_1 \;\;\text{and} \;\; u_0\in L^{p_1+1}(\Omega),
\end{align}
then the weak solution is global in time. As shown above, the logarithmic source term $f(u) = |u|^{p-1} u \log |u|$ satisfies (\ref{p1-1}) by setting $p_1 = p+\sigma$ for a sufficiently small $\sigma>0$. Moreover, since we assume $\min\{m,r\}>p$ and $u_0 \in L^{p+1+\sigma}(\Omega)$, condition (\ref{p1-3}) is satisfied. Therefore, the hypotheses of Theorem \ref{t:2} satisfy those of Theorem 1.5 in \cite{Becklin-Rammaha2}. Consequently, Theorem \ref{t:2} is a corollary of Theorem 1.5 in \cite{Becklin-Rammaha2}.

\vspace{0.1 in}

\section{Global Existence of Potential Well Solutions} \label{sec-Well}
In this section, we prove Theorem \ref{thm3-1}, which establishes the global existence of potential well solutions when the initial data $(u_0,w_0)\in W_1$. 
In our previous work \cite{FENG-JDE}, the source terms were required to be homogeneous for the potential well analysis. In the present manuscript, the logarithmic source terms are not homogeneous, and therefore our previous results in \cite{FENG-JDE} do not apply here. The nonhomogeneity of the logarithmic source terms causes difficulties in the analysis and constitutes a novel aspect of this work.

Before proving Theorem \ref{thm3-1}, we first prove Propositions \ref{lem3-1} and \ref{prop-W1}.

\subsection{Depth of the Potential Well}

In this subsection, we prove Proposition \ref{lem3-1}, which states that the depth $d$ of the potential well, as defined in \eqref{depth}, is strictly positive.

\begin{proof} [Proof of Proposition \ref{lem3-1}] 
Fix $(u,w) \in \mathcal{N}$. It follows from \eqref{3-8} and \eqref{Nehari} that
\begin{align}\label{3-11-1}
J(u,w)&\geq \frac{1}{2}(\|\nabla
u\|_2^2+|\Delta w|^2_2)- \frac{1}{p+1} \left(\int_\Omega |u|^{p+1} \log |u| dx +  \int_\Gamma |w|^{p+1} \log|w| d\Gamma \right) \notag\\
&\geq \left(\frac{1}{2}-\frac{1}{p+1}\right)(\|\nabla u\|_2^2+|\Delta w|^2_2).
\end{align}

Using $(u,w) \in \mathcal{N}$ and \eqref{ineq-log}, we have
\begin{align} \label{3-11-1-1}
\|\nabla u\|_2^2+|\Delta w|^2_2 &=            \int_\Omega |u|^{p+1} \log |u| dx + \int_\Gamma |w|^{p+1}\log|w| d\Gamma \notag\\
 &\leq  \int_{|u| > 1}|u|^{p+1+\sigma}|u|^{-\sigma}\log |u|dx + \int_{|w| > 1} |w|^{p+1+\sigma}|w|^{-\sigma} \log|w| d\Gamma\notag\\
 &\leq\frac{1}{e\sigma} \left(\|u\|^{p+1+\sigma}_{p+1+\sigma} + |w|^{p+1+\sigma}_{p+1+\sigma} \right) \notag\\
 &\leq C_{\sigma} \left( \|\nabla u\|^{p+1+\sigma}_{2} + |\Delta w|^{p+1+\sigma}_{2} \right),
\end{align}
where $\sigma>0$ is chosen sufficiently small so that $p+1+\sigma \leq 6$. 
Observe that, when $|u|<1$, the term $|u|^{p+1}\log |u|$ is negative and was therefore dropped from the above calculation. 

This gives
\begin{align}\label{3-11-2}
\|(u,w)\|^2_X\leq C_{\sigma}\|(u,w)\|^{p+1+\sigma}_X.
\end{align}
Since $(u,w)\neq(0,0)$, we obtain from \eqref{3-11-2} that
\begin{align} \label{3-11-2-2}
\|(u,w)\|_X\geq y_0        >0,
\end{align}
where $y_0 = \left(1/C_{\sigma}\right)^{\frac{1}{p-1+\sigma}}$.
Then, from \eqref{3-11-1} we obtain
$$
J(u,w)\geq \left(\frac{1}{2}-\frac{1}{p+1}\right)y_0^2 >0,\
\;\;\text{for all}\;\; (u,w)\in\mathcal{N}.
$$
Therefore, the depth $d$ of the potential well is positive, i.e., $d=\inf_{(u,w)\in\mathcal{N}}J(u,w) >0$.
This completes the proof.
\end{proof}

\vspace{0.1 in}

\subsection{The Stable Region $W_1$ of the Potential Well}
In this subsection, we prove Proposition \ref{prop-W1}, which states that the set $W_1$ defined in \eqref{defW1} is nontrivial; more precisely, it contains a ball centered at the origin.

\begin{proof}[Proof of Proposition \ref{prop-W1}]
Recall that $W=\{(u,w)\in X: J(u,w)<d\}$. Since $J(0,0) =0$ and $d>0$, we have $(0,0) \in W$. Furthermore, since $J$ is continuous, a neighborhood of $(0,0)$ in $X$ is contained in $W$;
that is, there exists $r>0$ such that the ball $B_{r}(0,0) = \{(u,w) \in X: \|(u,w)\|_X < r\}\subset W$. 
By estimates similar to those in (\ref{3-11-1-1})-(\ref{3-11-2-2}), we have $\|(u,w)\|_X\geq y_0>0$ for all $(u,w) \in W_2$. Thus, choosing $r< y_0$, we have $B_r(0,0) \cap W_2 =  \emptyset $.
Since $W_1 \cap W_2 = \emptyset$ and $W_1 \cup W_2 = W$, we conclude that $B_r(0,0) \subset W_1$.

\end{proof}

\subsection{Global Existence with $W_1$ Initial Data}
In this subsection, we prove Theorem \ref{thm3-1}, which establishes the global existence of weak solutions for initial data $(u_0,w_0)\in W_1$, the stable region of the potential well. 

\begin{proof} [Proof of Theorem \ref{thm3-1}]
Let $(u_0,w_0)\in W_1$ and $\widehat{E}(0)<d$.
The local well-posedness of the weak solution $(u(t),w(t))$ on $[0,T)$ is guaranteed by Theorem \ref{t:1}, where $[0,T)$ is the maximal interval of existence.

The energy identity (\ref{energy}) can be written as
\begin{align} \label{3-15}
\widehat{E}(t)+\int^t_0\int_\Omega
|u_t|^{m+1} dx ds 
+ \int^t_0\int_\Gamma |w_t|^{r+1} d\Gamma ds = \widehat{E}(0),   \;\;  \text{for all}   \,\, t\in [0,T).
\end{align}
Therefore, we obtain $\widehat{E}(t)\leq\widehat{E}(0)<d$.
In addition, since $J(u(t),w(t))\leq \widehat E(t)$, we have for any $t\in [0,T)$,
\begin{align}\label{3-16}
J(u(t),w(t))\leq\widehat{E}(t)\leq\widehat{E}(0)<d.
\end{align}
This means that assertion ($i$) holds, and we infer that $(u(t),w(t))\in W$ for all $t\in [0,T)$.

Next we prove that assertion ($ii$) holds. Namely, if   $(u_0,w_0)\in W_1$, 
the solution trajectory $(u(t),w(t))$ belongs to $W_1$ during the entire lifespan of the solution. 

We argue by contradiction to prove that $(u(t),w(t))\in W_1$ for all $t\in [0,T)$.
Suppose that there exists a time $t_1\in(0,T)$ such that
$(u(t_1),w(t_1))\notin W_1$. 
Since $(u(t_1),w(t_1)) \in W$, and 
$W_1\cup W_2=W$, $W_1\cap W_2=\emptyset$, it follows that $(u(t_1),w(t_1))\in W_2$.

Next, we show that $I(u(t),w(t))$ is continuous in $t$. Here we recall that $I(u,w) =   \|\nabla u\|^2_2+|\Delta w|^2_2-\int_\Omega |u|^{p+1} \log |u| dx-\int_\Gamma |w|^{p+1} \log |w| d\Gamma$.

Note that $\frac{d}{ds}(|s|^{p+1}\log |s|) = |s|^{p-1} s \left((p+1)\log|s|+1\right)$. Then, for any $t$, $t_0\in [0,T)$, by using the Mean Value Theorem, the assumption $1<p<5$ and (\ref{ineq-log}), we obtain 
\begin{align}  \label{G1} 
&\left|\int_\O|u(t)|^{p+1}\log|u(t)|dx - \int_\O|u(t_0)|^{p+1}\log|u(t_0)|dx\right|\notag\\
&\quad \leq C  \int_{\O} \Big|  |u(t)|^{5}  +    |u(t_0)|^{5} + 1 \Big|    |u(t)-u(t_0)|dx \notag\\
&\quad \leq C \left(\|u(t)\|_6^5 + \|u(t_0)\|_6^5 +1 \right) \|u(t) - u(t_0)\|_6 \notag\\
&\quad \leq C  \left(\|\nabla u(t)\|_{2}^{5} + \| \nabla u(t_0)\|_{2}^{5} +1 \right) \|\nabla( u(t) - u(t_0))\|_2 \rightarrow 0, \;\text{as} \; t\rightarrow t_0,
\end{align}
because $u \in C([0,T); H^1_{\G_0}(\O))$. It follows that $t\mapsto \int_{\Omega} |u|^{p+1} \log|u| dx$ is continuous on $[0,T)$.

Similarly, by using $w \in C([0,T);H^2_0(\Gamma))$, we obtain that $t\mapsto \int_\Gamma |w|^{p+1} \log|w| d\G$ is also continuous on $[0,T)$.

Consequently, $I(u(t),w(t))$ is continuous in $t$ on $[0,T)$.

Since $(u(0),w(0))\in W_1$, $(u(t_1),w(t_1))\in W_2$, and $I(u(t),w(t))$ is continuous in $t$, the Intermediate Value Theorem implies that there exists a time $s\in [0,t_1)$ such that
$I(u(s),w(s)) = 0$. We denote by $t^*$ the supremum of all $s \in [0,t_1)$ satisfying $I(u(s),w(s)) = 0$.
Thanks to the continuity of the function $I(u(t),w(t))$, we have that $t^* \in [0,t_1)$ satisfies $I(u(t^*),w(t^*)) = 0$, and $(u(t),w(t))\in W_2$
for any $t\in (t^*,t_1]$. We consider the following two cases.

{\bf Case 1.}  $(u(t^*),w(t^*))\neq(0,0)$.

Recalling the definition of the Nehari manifold $\mathcal N$ and noting that $I(u(t^*),w(t^*)) = 0$, we have $(u(t^*),w(t^*))\in\mathcal{N}$.
Then we obtain from \eqref{depth} that $J(u(t^*),w(t^*))\geq d$. This  contradicts \eqref{3-16}.

{\bf Case 2.} $(u(t^*),w(t^*))=(0,0)$.

 Note that
$(u(t),w(t))\in W_2$ for any $t\in(t^*,t_1]$. In view of the definition of the set $W_2$ and using estimates similar to those in (\ref{3-11-1-1})-\eqref{3-11-2}, one has that for any $t\in(t^*,t_1]$,
\begin{align}   \label{uwX}
\|(u(t),w(t))\|^2_X < C_{\sigma} \|(u(t),w(t))\|^{p+1+\sigma}_X ,
\end{align}
where $\sigma>0$ is small enough so that $p+1+\sigma \leq 6$. Since $(0,0)\notin W_2$, we have $(u(t),w(t)) \not=(0,0)$ for any $t\in (t^*,t_1]$.
Then, (\ref{uwX}) shows that $\|(u(t),w(t))\|_X > y_0 =     \left(1/C_{\sigma}\right)^{\frac{1}{p-1+\sigma}} >0$, for any $t\in (t^*,t_1]$.
By using the continuity of the weak solution $(u(t),w(t))$ from $[0,T)$ to $X$, we obtain
$\|(u(t^*),w(t^*))\|_X\geq y_0>0$, which contradicts $(u(t^*),w(t^*))=(0,0)$. 

Therefore $(u(t),w(t))\in W_1$ for all $t\in [0,T)$. This completes the proof of assertion ($ii$).

Finally, we prove the global existence of the weak solution $(u(t),w(t))$.  
It suffices to show that the quadratic energy $E(t)$ has a uniform bound independent of time.

By the definition of $W_1$, we have $I(u,w) \geq 0$. Then,
\begin{align}\label{E1}
\widehat E(t) \geq &\frac{1}{2}(\|u_t\|^2_2+|w_t|^2_2)+\left(\frac{1}{2}-\frac{1}{p+1}\right)(\|\nabla u\|^2_2 + |\Delta w|_2^2) \notag\\
&+ \frac{1}{(p+1)^2}\left(\|u\|^{p+1}_{p+1} + |w|^{p+1}_{p+1}\right).
\end{align}
By (\ref{E1}) and (\ref{3-16}), we obtain that for any $t\in[0,T)$,
\begin{align} \label{ubE}
E(t)&\leq \frac{p+1}{p-1} \widehat E(t) < \frac{p+1}{p-1} d.
\end{align}
This proves assertion ($iii$).

By Theorem \ref{t:1}, the local weak solution exists on $[0,T_0]$ 
where $T_0$ depends on the initial quadratic energy $E(0)$. Due to the uniform bound (\ref{ubE}) for the quadratic energy, 
the local solution can be extended to a global solution. Namely, the maximum lifespan $T=\infty$. 
This completes the proof.
\end{proof}

\vspace{0.1 in}

\begin{rmk}\label{rmk-4-1}
The continuity of the solution $(u,w)$ mapping from $[0,T)$ to $X=H^1_{\Gamma_0}(\Omega)\times H^2_0(\Gamma)$ is critical for the argument above, and essential for the validity of the entire paper. For instance, the implementation of the Intermediate Value Theorem in the above proof depends on the continuity of the solution. 
\end{rmk}

\vspace{0.1 in}

\section{Energy Decay Rates} \label{sec-decay}

In this section, we establish Theorem \ref{thm4-1}, namely, the uniform energy decay rates of potential well solutions.
We first establish a stabilization estimate, which is essential for proving the energy decay rates.

\subsection{A Stabilization Estimate}
We denote
\begin{align}    \label{Dt}
D(T):=\int^T_0(\|u_t(t)\|^{m+1}_{m+1}+|w_t(t)|^{r+1}_{r+1})dt.
\end{align}
Then the energy identity \eqref{3-15} becomes
\begin{align}\label{en-id}
\widehat{E}(t)+D(t)=\widehat{E}(0).
\end{align}

\begin{lem}\label{lem4-1}
Assume that the parameters \( m \), \( r \), and \( p \) satisfy \eqref{p}, with \( 1 \leq m \leq 5 \). 
Fix an arbitrary $d_0 \in (0,d)$, where $d$ is the depth of the potential well defined in (\ref{depth}). Let \( (u_0, w_0) \in W_1 \) with \( \widehat{E}(0) \leq d_0 <d \). Then the global potential well solution of system \eqref{PDE}, established in Theorem \ref{thm3-1}, satisfies
\begin{align}\label{4-12}
\widehat{E}(T)\leq  C_0   \Big( [D(T)]^{\frac{2}{m+1}}+ [D(T)]^{\frac{2}{r+1}} + D(T)\Big),
\end{align}
for \( T \) sufficiently large, where both the required size of $T$ and the constant \( C_0 \) depend on the ratio \( d/d_0 \). 
\end{lem}

\begin{rmk}
By \eqref{p} and \(1\leq m\leq 5\), we have \(1< p < 5\) in Lemma \ref{lem4-1}. Additionally, the constant \( C_0 \) may increase as \( d_0 \) approaches \( d \), that is, as the ratio \( d/d_0 \) approaches \( 1 \).
\end{rmk}

\begin{proof}
Since $u \in C([0,T];H^1_{\Gamma_0}(\Omega))$ and $1\leq m\leq 5$, we have $u \in L^{m+1}(\Omega\times(0,T))$. 
Therefore, we can replace $\phi$ with $u$ in \eqref{wkslnwave} and $\psi$ with $w$ in \eqref{wkslnplt},
and add the resulting equations to obtain
\begin{align}    \label{replace}
&\int^T_0(\|\nabla
u\|^2_2+|\Delta w|^2_2)dt-\int^T_0(\|u_t\|^2_2+|w_t|^2_2)dt+\int_\Omega u_tudx\bigg|^T_0+\int_\Gamma(w_tw+\gamma
uw)d\Gamma\bigg|^T_0\nonumber\\
&\qquad-2\int^T_0\int_\Gamma\gamma uw_td\Gamma
dt+\int^T_0\int_\Omega |u_t|^{m-1}u_tudxdt+\int^T_0\int_\Gamma
|w_t|^{r-1} w_t w d\Gamma dt\nonumber\\
&\qquad=\int^T_0\int_\Omega |u|^{p+1}\log|u| dxdt+\int^T_0\int_\Gamma
|w|^{p+1}\log|w| d\Gamma dt. 
\end{align}
Multiplying $(\ref{replace})$ by $\frac{1}{2}$ and using the definition of $\widehat{E}(t)$ in \eqref{3-9}, we obtain
\begin{align}\label{4-13-1}
\int^T_0\widehat E(t)dt&=\underbrace{-\frac{1}{2}\int_\Omega
u_tudx\bigg|^T_0-\frac{1}{2}\int_\Gamma (w_t w+\gamma u w)
d\Gamma\bigg|^T_0}_{:=Y_1}+\underbrace{\int^T_0(\|u_t\|^2_2+|w_t|^2_2)dt}_{:=Y_2}\nonumber\\
&\quad\underbrace{+\int^T_0\int_\Gamma \gamma u w_td\Gamma
dt}_{:=Y_3} \, \underbrace{-\frac{1}{2}\int^T_0\int_\Omega
|u_t|^{m-1}u_tudxdt-\frac{1}{2}\int^T_0\int_\Gamma |w_t|^{r-1} w_t w d\Gamma
dt}_{:=Y_4}\nonumber\\
&\quad\underbrace{+\left(\frac{1}{2}-\frac{1}{p+1}\right)\left(\int^T_0\int_\Omega
|u|^{p+1} \log|u| dxdt + \int^T_0\int_\Gamma |w|^{p+1} \log|w| d\Gamma dt \right)      }_{:=Y_{51}}\notag\\
&\quad\underbrace{+\frac{1}{(p+1)^2}\int^T_0 \left(\|u\|^{p+1}_{p+1} + |w|^{p+1}_{p+1} \right) dt}_{:=Y_{52}}.
\end{align}

We now estimate each term on the right-hand side of \eqref{4-13-1}. 

\smallskip

{\bf (1)  Estimate for $Y_1$.} Using the Cauchy-Schwarz inequality and the trace embedding, we obtain
\begin{align}\label{4-13-2}
&\left|\int_\Omega u_tudx+\int_\Gamma(w_tw+\gamma u
w)d\Gamma\right|\notag\\
&\quad\leq\frac{1}{2}(\|u_t\|^2_2+\|u\|^2_2+|w_t|^2_2+2|w|^2_2+|\gamma
u|^2_2)\nonumber\\
&\quad\leq C\Big(\|u_t\|^2_2+|w_t|^2_2+\|\nabla
u\|^2_2+|\Delta w|^2_2\Big) \leq C E(t) \leq C\widehat E(t),
\end{align}
where the last inequality is due to item ($iii$) of Theorem \ref{thm3-1}.

Then we infer from  \eqref{4-13-2} and \eqref{en-id} that
\begin{align}\label{4-14}
Y_1\leq C(\widehat E(T)+\widehat E(0))\leq C(2\widehat{E}(T)+D(T)).
\end{align}

\vspace{0.1 in}

{\bf (2) Estimate for $Y_2$.}  

Using  H\"{o}lder's inequality, we obtain
$$
\int^T_0\|u_t\|^2_2dt\leq (T|\O|)^{\frac{m-1}{m+1}}\left(\int^T_0\|u_t\|^{m+1}_{m+1} dt\right)^{\frac{2}{m+1}}\leq T|\O|  [D(T)]^{\frac{2}{m+1}},
$$
provided that $T$ is sufficiently large so that $T |\O| \geq 1$.
Similarly, we have
\begin{align}\label{E3}
\int^T_0|w_t|^2_2dt\leq T|\G|    [D(T)]^{\frac{2}{r+1}},
\end{align}
if $T |\Gamma| \geq 1$. Then we infer that
\begin{align}\label{E2}
Y_2\leq T(|\O|+|\G|)\Big( [D(T)]^{\frac{2}{m+1}} + [D(T)]^{\frac{2}{r+1}} \Big).
\end{align}

\vspace{0.1 in}

{\bf (3) Estimate for $Y_3$.} 

Using \eqref{E1} and \eqref{E3}, we obtain 
\begin{align}\label{E4}
Y_3 &\leq  \int_0^T |\gamma u|_2 |w_t|_2  dt  \leq  C\int_0^T \|\nabla u\|_2 |w_t|_2  dt  \leq  C\int_0^T  \left(\widehat E(t)\right)^{1/2}       |w_t|_2  dt  \notag\\
&\leq   \varepsilon \int^T_0   \widehat E(t)     dt+C_{\varepsilon}\int^T_0|w_t|^2_2 dt \leq \varepsilon\int^T_0   \widehat E(t)dt+TC_{\varepsilon}|\G| [D(T)]^{\frac{2}{r+1}}.
\end{align}

\vspace{0.1 in}

{\bf (4) Estimate for $Y_4$.}

By H\"{o}lder's inequality, we have
\begin{align*}
\left|\int^T_0\int_\O|u_t|^{m-1} u_t u dxdt \right|&\leq\left(\int^T_0\int_\O |u_t|^{m+1} dxdt \right)^{\frac{m}{m+1}} \left(\int^T_0\int_\O |u|^{m+1} dxdt \right)^{\frac{1}{m+1}}\\
&\leq [D(T)]^{\frac{m}{m+1}} \|u\|_{L^{m+1}(\O\times(0,T))}.
\end{align*}

Since $1\leq m \leq 5$, we have $\|u\|_{m+1} \leq  C\|\nabla u\|_2$ due to the Sobolev embedding. 
Moreover, from (\ref{E1}), we have $\|\nabla u\|_2^2 \leq C \widehat E(t)$.
In addition, by \eqref{3-15}, $\widehat E(t) \leq \widehat E(0)<d$.
Therefore, we obtain
\begin{align*}
\|u\|_{L^{m+1}(\O\times(0,T))}^{m+1} &\leq C\int^T_0\|\nabla u\|^{m+1}_2dt
\leq C \int_0^T \Big(\widehat E(t) \Big)^{\frac{m+1}{2}} dt
\leq C  d^{\frac{m-1}{2}}\int^T_0\widehat E(t)dt.
\end{align*}
It follows that
\begin{align}\label{E5}
\left|\int^T_0\int_\O|u_t|^{m-1} u_t u dxdt\right|\leq C [D(T)]^{\frac{m}{m+1}}\left(\int^T_0\widehat E(t)dt\right)^{\frac{1}{m+1}}.
\end{align}

Applying the Young inequality to \eqref{E5}, we obtain that for any $\varepsilon>0$,
$$
 \left|\int^T_0\int_\O|u_t|^{m-1} u_t u dxdt \right|\leq \varepsilon\int^T_0\widehat E(t)dt+C_{\varepsilon} D(T).
 $$
We can use a similar argument to estimate the second term of $Y_4$. As a result, for any $\varepsilon>0$,
 \begin{align}\label{E6}
 Y_4\leq 2\varepsilon\int^T_0\widehat E(t)dt+C_{\varepsilon} D(T).
 \end{align}

\vspace{0.1 in}

 {\bf (5) Estimate for $Y_5:=Y_{51}+Y_{52}$}.
This estimate is critical for the proof of the stabilization estimate (\ref{4-12}). 
We first claim that
 \begin{align}\label{E7}
 I(u,w) \geq  \left(1 -    \left( \frac{d}{d_0} \right)^{\frac{1-p}{p+1}}    \right )(\|\nabla u\|^2_2+|\Delta w|^2_2) +\frac{1}{p+1}\log\left( \frac{d}{d_0} \right) (\|u\|^{p+1}_{p+1}+|w|^{p+1}_{p+1}).
 \end{align}
Indeed, since $(u,w) \in W_1$, either $I(u,w)>0$ or $(u,w) = (0,0)$. Clearly, (\ref{E7}) holds for the case $(u,w) = (0,0)$.
It remains to consider the case $I(u,w)>0$. For any $\lambda>0$, we have 
\begin{align} \label{scaleI}
I(\lambda(u,w)) =  & \lambda^2 (\|\nabla u\|^2_2+|\Delta w|^2_2) - \lambda^{p+1} \left(\int_\Omega |u|^{p+1} \log |u| dx + \int_\Gamma |w|^{p+1} \log |w| d\Gamma\right) \notag\\
& -  \lambda^{p+1} \log \lambda  \left(    \|u\|_{p+1}^{p+1} +   |w|_{p+1}^{p+1} \right)   .
\end{align}
Note that, when $\lambda \rightarrow \infty$, the leading order in (\ref{scaleI}) is $\lambda^{p+1} \log \lambda$. 
Hence, $$\lim_{\lambda \rightarrow \infty} I(\lambda(u,w)) = -\infty.$$ Then, since $I(u,w)>0$, 
we use the Intermediate Value Theorem to conclude that there exists at least one real number $\lambda_0>1$ such that $I(\lambda_0(u,w))=0$. Therefore, $\lambda_0(u,w)$ belongs to the Nehari manifold $\mathcal N$, and thus $J(\lambda_0(u,w)) \geq d$ by (\ref{depth}).
We remark that the value of $\lambda_0$ depends on $(u,w)$. 

It follows that
\begin{align} \label{E7-1}
d&\leq J(\lambda_0(u,w)) \notag\\
&=\frac{1}{p+1}I(\lambda_0(u,w))+ \Big(\frac{1}{2} - \frac{1}{p+1}\Big) (\|\nabla(\lambda_0 u)\|^2_2+|\Delta(\lambda_0 w)|^2_2)    + \frac{  \|\lambda_0 u\|^{p+1}_{p+1}    }{(p+1)^2} 
+ \frac{|\lambda_0 w|^{p+1}_{p+1}}{(p+1)^2} \notag\\
&=\lambda_0^{p+1}\Big[ \Big(\frac{1}{2} - \frac{1}{p+1}\Big)\lambda^{1-p}_0 (\|\nabla u\|^2_2+|\Delta w|^2_2)  + \frac{ \|u\|^{p+1}_{p+1}  }{(p+1)^2} + \frac{|w|^{p+1}_{p+1}}{(p+1)^2} \Big] \notag\\
&\leq \lambda_0^{p+1}\Big[ \Big(\frac{1}{2} - \frac{1}{p+1}\Big) (\|\nabla u\|^2_2+|\Delta w|^2_2)  + \frac{ \|u\|^{p+1}_{p+1}  }{(p+1)^2} + \frac{|w|^{p+1}_{p+1}}{(p+1)^2} \Big],
\end{align} 
where the last inequality is due to the fact that $\lambda_0 > 1$ and $p>1$.

Combining (\ref{E1}) and (\ref{E7-1}) yields  
$$d\leq \lambda^{p+1}_0\widehat E(t)  \leq \lambda^{p+1}_0\widehat E(0) \leq    \lambda^{p+1}_0 d_0    .$$
Consequently,
\begin{align}\label{E8}
\lambda_0 \geq \left(d/d_0\right)^{\frac{1}{p+1}}>1,
\end{align}
since $0<d_0<d$. We remark that (\ref{E8}) shows that, although the value $\lambda_0$ depends on $(u,w)$, it has a uniform lower bound.  

By the definition of $I(u,w)$, it follows that
\begin{align*}
0&=I(\lambda_0(u,w))=\lambda_0^2(\|\nabla u\|^2_2+|\Delta w|^2_2)-\lambda_0^{p+1} \left(\int_\O|u|^{p+1}\log|u|dx  +   \int_\G|w|^{p+1} \log|w|d\G \right)\\
&\qquad\qquad\qquad\quad\ -\lambda_0^{p+1} \log\lambda_0(\|u\|^{p+1}_{p+1}+|w|^{p+1}_{p+1})\\
&=\lambda_0^{p+1} I(u,w)-(\lambda^{p+1}_0-\lambda_0^2)(\|\nabla u\|^2_2+|\Delta w|^2_2)-\lambda_0^{p+1} \log\lambda_0(\|u\|^{p+1}_{p+1}+|w|^{p+1}_{p+1}).
\end{align*}
Together with \eqref{E8}, this gives
\begin{align*}
I(u,w)&=(1-\lambda_0^{1-p})(\|\nabla u\|^2_2+|\Delta w|^2_2)+ \log\lambda_0(\|u\|^{p+1}_{p+1}+|w|^{p+1}_{p+1})\\
&\geq \left( 1-  \left(d/d_0\right)^{\frac{1-p}{p+1}} \right)(\|\nabla u\|^2_2+|\Delta w|^2_2)+\frac{1}{p+1}\log\left( d / d_0 \right)(\|u\|^{p+1}_{p+1}+|w|^{p+1}_{p+1}),
\end{align*}
i.e., the inequality \eqref{E7} holds.

Using (\ref{3-9-0}) and \eqref{E7}, we obtain
\begin{align}\label{E9}
&\int_\O|u|^{p+1} \log|u| dx+\int_\G|w|^{p+1} \log|w| d\G\notag\\
&\quad \leq     \left(d/d_0 \right)^{\frac{1-p}{p+1}}         (\|\nabla u\|^2_2+|\Delta w|^2_2)
-\frac{1}{p+1}\log\left( d/d_0 \right)(\|u\|^{p+1}_{p+1}+|w|^{p+1}_{p+1}).
\end{align}
By (\ref{E1}) and \eqref{E9}, we obtain
\begin{align}\label{E10}
Y_5&:=Y_{51}+Y_{52}\notag\\
& \leq  \left(\frac{1}{2}-\frac{1}{p+1}\right)    \left(d/d_0\right)^{\frac{1-p}{p+1}}         \int^T_0(\|\nabla u\|^2_2+|\Delta w|^2_2)dt\notag\\
&\quad+\frac{1}{(p+1)^2}\left[1- \frac{p-1}{2}\log\left( d/d_0 \right)\right]
\int^T_0 \left(\|u\|^{p+1}_{p+1}  + |w|^{p+1}_{p+1} \right) dt \notag\\
&\leq \theta\int^T_0\widehat E(t)dt,
\end{align}
where 
\begin{align}    \label{def-theta}
\theta=\max\left\{     \left(d/d_0\right)^{\frac{1-p}{p+1}}       , \;\; 1- \frac{p-1}{2}\log\left( d/d_0 \right)\right\}.
\end{align}
Since $p>1$ and $d>d_0$, we have $\theta<1$.

We insert \eqref{4-14}, \eqref{E2}, \eqref{E4}, \eqref{E6} and \eqref{E10} into \eqref{4-13-1} and choose $\varepsilon>0$ sufficiently small so that $3 \varepsilon + \theta \leq 1$ to obtain
\begin{align} \label{bE11}
\int^T_0\widehat E(t)dt
\leq C_{\varepsilon} \Big(T [D(T)]^{\frac{2}{m+1}}+T [D(T)]^{\frac{2}{r+1}}+\widehat E(T)+D(T)\Big).
\end{align}

From (\ref{def-theta}), $\theta$ depends on the ratio $d/d_0$. Since $3\varepsilon + \theta \leq 1$, it follows that $\varepsilon$ depends on $d/d_0$. 
Thus, $C_{\varepsilon}$ in (\ref{bE11}) depends on $d/d_0$.

Now, by (\ref{3-15}), we have $\widehat E'(t) \leq 0$, i.e., $\widehat E(t)$ is non-increasing. Hence, from (\ref{bE11}) we obtain
\begin{align} \label{E11}
T\widehat E(T)\leq \int^T_0\widehat E(t)\,dt     \leq        C_{\varepsilon} \Big(T [D(T)]^{\frac{2}{m+1}}+T [D(T)]^{\frac{2}{r+1}} +\widehat E(T)+D(T)\Big).
\end{align}
Choosing $T>2C_{\varepsilon}$ in \eqref{E11}, we have
$$
T\widehat E(T)\leq  C_{\varepsilon} \Big(T [D(T)]^{\frac{2}{m+1}}+T [D(T)]^{\frac{2}{r+1}} + D(T) \Big).
$$
Dividing both sides of the above inequality by $T\geq 1$, we obtain
\begin{align}\label{E12}
\widehat E(T)\leq  C_{\varepsilon} \Big([D(T)]^{\frac{2}{m+1}}+[D(T)]^{\frac{2}{r+1}} +D(T)\Big),
\end{align}
for sufficiently large $T$, depending on $\varepsilon$, which depends on the ratio $d/d_0$. This completes the proof of Lemma \ref{lem4-1}.
\end{proof}

\begin{rmk}
We emphasize that $\theta$ defined in (\ref{def-theta}) being strictly less than 1 is crucial for our argument. 
Indeed, the right-hand side of \eqref{E10} can be completely absorbed by the term $\int_0^T \widehat E(t) dt$ on the left-hand side of \eqref{4-13-1}. 
This makes the proof concise since there are no lower-order terms in the stabilization estimate. 
\end{rmk}

\vspace{0.1 in}

\subsection{Proof of Energy Decay Rates}

We now complete the proof of Theorem \ref{thm4-1} by using the stabilization estimate \eqref{E12}. The strategy of the proof is adapted from the paper \cite{LTa} by Lasiecka and Tataru.
The idea is to relate the stabilization estimate to an ODE, whose decay rate then determines the energy decay rate of the PDE.

\begin{proof}[Proof of Theorem \ref{thm4-1}] 
We define a function $\Phi(s): [0,\infty)  \rightarrow [0,\infty) $ by 
\begin{align}   \label{Phi}
 \Phi(s) :=C_0 (s^{\frac{2}{m+1}}+s^{\frac{2}{r+1}} + s).
\end{align}
Note that $\Phi(s)$ is monotone increasing and vanishing at the origin.
Then it follows from \eqref{4-12} and \eqref{en-id} that
\begin{align}\label{4-69}
\widehat{E}(T)\leq  \Phi(D(T))=  \Phi(\widehat{E}(0)-\widehat{E}(T)).
\end{align}
We infer from \eqref{4-69} that
\begin{align}\label{4-70}
(I + \Phi^{-1}) \widehat{E}(T) \leq \widehat{E}(0),
\end{align}
for $T>0$ sufficiently large. 

We can iterate \eqref{4-70} to obtain
\begin{align}\label{E13}
 (I  +    \Phi^{-1})  \widehat{E}((k+1)T) \leq\widehat{E}(kT),\
\ k=0,1,2,...
\end{align}
Next, we claim that 
\begin{align}\label{4-71}
\widehat{E}(kT)\leq \mathcal S(k),\ \ k=0,1,2,...,
\end{align}
where $\mathcal S(t)$ is the solution of the ODE
\begin{align}\label{4-73}
\begin{cases}
    \mathcal S'(t)+(I+ \Phi)^{-1}\mathcal S(t)=0, \\
    \mathcal S(0)=\widehat{E}(0).
 \end{cases}
\end{align}
We prove \eqref{4-71} by induction. Indeed, since $\mathcal S(0)=\widehat E(0)$, \eqref{4-71} holds for $k=0$. For a given $k\geq 0$, we assume $\widehat E(kT)\leq \mathcal S(k)$ and show $\widehat E((k+1)T)\leq \mathcal S(k+1)$. It follows directly from \eqref{E13} that
\begin{align}\label{E14}
\widehat E((k+1)T)\leq (I+\Phi^{-1})^{-1}\widehat E(kT).
\end{align}
Note that $(I+\Phi)^{-1}$ is increasing on $[0,\infty)$ and vanishing at the origin. From \eqref{4-73} we infer that $\mathcal S(t)$ is decreasing on $[0,\infty)$ 
and $\lim_{t\to\infty}\mathcal S(t)=0$. Hence we integrate \eqref{4-73} from $k$ to $k+1$ to obtain
\begin{align}\label{E15}
\mathcal S(k)-(I+\Phi)^{-1}\mathcal S(k)\leq \mathcal S(k)-\int^{k+1}_k (I+\Phi)^{-1}\mathcal S(t)dt = \mathcal S(k+1).
\end{align}
It is straightforward to verify that 
$I-(I+\Phi)^{-1}=(I+\Phi^{-1})^{-1}$, and thus \eqref{E15} shows that $(I+\Phi^{-1})^{-1} \mathcal S(k)\leq \mathcal S(k+1)$. 
Using the increasing property of $(I+\Phi^{-1})^{-1}$ and the induction hypothesis, we obtain
\begin{align}\label{E16}
(I+\Phi^{-1})^{-1}\widehat E(kT)\leq (I+\Phi^{-1})^{-1} \mathcal S(k)\leq \mathcal S(k+1).
\end{align}
Then it follows from \eqref{E14} and \eqref{E16} that $\widehat E((k+1)T)\leq \mathcal S(k+1)$. This proves \eqref{4-71}.

For any $t>T$, write
$t=kT+b$ with  $k\in\mathbb{N}$ and $0\leq b<T$, so that 
$k=\frac{t}{T}-\frac{b}{T}>\frac{t}{T}-1$. Since
$\widehat{E}(t)$ and $\mathcal S(t)$ are monotone decreasing, we obtain from \eqref{4-71} that for any $t>T$,
\begin{align}\label{4-74}
\widehat{E}(t)=\widehat{E}(kT+b)\leq\widehat{E}(kT)\leq
\mathcal S(k)\leq \mathcal S\left(\frac{t}{T}-1\right).
\end{align}

{\bf Case 1. $m=r=1$.}

In this case, $\Phi(s)$ is linear. Hence the ODE \eqref{4-73} is also linear, i.e., for a constant $\delta>0$,
\begin{align}
\begin{cases}
    \mathcal S'(t)+\delta \mathcal S(t)=0, \\
    \mathcal S(0)=\widehat{E}(0),
\end{cases}\nonumber
\end{align}
which implies
$$
\mathcal S(t)=\widehat{E}(0)e^{-\delta t}.
$$
It follows from \eqref{4-74} that for any $t>T$,
\begin{align}\label{4-75}
\widehat{E}(t)\leq
\widehat{E}(0)e^{-\delta(\frac{t}{T}-1)}= e^{\delta}\widehat{E}(0) e^{-\frac{\delta}{T}t}.
\end{align}
Then \eqref{exp} follows by setting $a=\frac{\delta}{T}$ in  \eqref{4-75}. 

\smallskip

{\bf Case 2. $m>1$ or $r>1$.}

In this case, for $0\leq s\leq 1$,
$$
\Phi(s) = C_0 (s^{\frac{2}{m+1}}+s^{\frac{2}{r+1}} + s)\leq 3C_0 s^{\frac{2}{\eta +1}},
$$
where $\eta=\max\{m,r\}$. Since $\lim_{t\to\infty}\mathcal S(t)=0$, we infer from \eqref{4-73} that for $t$ sufficiently large,
\begin{align*}
\begin{cases}
\mathcal S'(t)+C(\widehat E(0)) \mathcal S^{\frac{\eta +1 }{2}}(t)\leq 0,\\
\mathcal S(0)=\widehat E(0),
\end{cases}
\end{align*}
which yields
$$
\mathcal S(t)\leq C(\widehat E(0))(1+t)^{-\frac{2}{\eta-1}},\ \ \mbox{for}\ t\geq 0,
$$
where $C(\widehat E(0))$ depends on $\widehat E(0)$.

Consequently, by \eqref{4-74}, we obtain for $t\geq 0$,
$$
\widehat E(t)\leq C(\widehat E(0))(1+t)^{-\frac{2}{\eta-1}}.
$$
This completes the proof of Theorem \ref{thm4-1}.
\end{proof}

\vspace{0.1 in}

\section{Blow-Up of Weak Solutions} \label{sec-Blow}

This section is devoted to proving that the weak solution of system \eqref{PDE} blows up in finite time, provided that the source terms surpass the damping terms. We treat two cases separately: the negative initial energy case and the positive initial energy case.

\subsection{Blow-Up of Solutions with Negative Initial Energy}\label{blowup1}
In this subsection, we prove Theorem \ref{thm6-1}, which states that the weak solution of system \eqref{PDE} blows up in finite time when the source terms are stronger than the damping terms and the initial total energy \(\widehat E(0)\) is negative.

\begin{proof}[Proof of Theorem \ref{thm6-1}]
Let $(u(t), w(t))$ be a weak solution of \eqref{PDE}. 
We define the \emph{maximal life span} $T$ of such a solution to be the supremum of all $T^* > 0$ such that 
$(u(t), w(t))$ remains a solution to system \eqref{PDE} on $[0, T^*]$. By the definition of weak solutions (see Definition~\ref{def:weaksln}),
weak solutions are continuous in time with values in the energy space. 
Hence, the quadratic energy $E(t)$ is continuous on $[0, T)$. Moreover, by the local well-posedness result (Theorem~\ref{t:1}), 
the local existence time depends on the initial quadratic energy $E(0)$. Therefore, an \emph{a priori} bound on the quadratic energy allows one to extend a local solution to a global solution. On the other hand, if the maximal life span satisfies $T < \infty$, then
\begin{align}\label{6-33}
\limsup_{t \to T^-} E(t) = +\infty.
\end{align}

In what follows, we prove that $T$ is finite and derive an upper bound for $T$.

Motivated by \cite{ACCRT,BL3,GR1}, for any $t\in[0,T)$, we define
\begin{align} \label{SS}
G(t)=-\widehat{E}(t),\ \
S(t)= \int_\Omega |u|^{p+1} \log|u|dx + \int_\Gamma |w|^{p+1} \log |w|d\Gamma,
\end{align}
where  $\widehat E(t)$  is the total energy. Then, from \eqref{3-15}, we infer that
\begin{align}\label{B1}
G'(t)=-\widehat E'(t)=\|u_t\|^{m+1}_{m+1}+|w_t|^{r+1}_{r+1}\geq 0.
\end{align}
This yields
\begin{align}\label{B2}
G(t) \geq G(0)=-\widehat E(0)>0,\ \ \mbox{for}\ t>0.
\end{align}
Note that
\begin{align}\label{B3}
\frac{1}{p+1}S(t)=G(t)+E(t)+\frac{1}{(p+1)^2}(\|u\|^{p+1}_{p+1}+|w|^{p+1}_{p+1}),
\end{align}
and hence, for any $t>0$,
\begin{align}\label{B4}
\frac{1}{p+1}S(t)\geq G(t) \geq G(0)>0.
\end{align}

Since $p>m$ and $p>r>1$, we introduce a positive constant $a$ satisfying
\begin{align}\label{6-4}
0<a<\min\left\{\frac{1}{m+1}-\frac{1}{p+1},
\;\;\frac{1}{r+1}-\frac{1}{p+1},
\;\;\frac{p-1}{2(p+1)},    \;\;  \frac{1}{2} -  \frac{1}{r+1}   \right\}.
\end{align}

We define
\begin{align}  \label{def-Q}
Q(t)=(u(t),u_t(t))_\Omega+(w(t),w_t(t))_\G+(\gamma u(t),w(t))_\G.
\end{align}
As in \cite{GT}, we define
\begin{align} \label{def-Y}
\mathcal{Y}(t):=G^{1-a}(t)+\varepsilon Q(t),
\end{align}
where $0<\varepsilon\leq\min\{1,G(0)\}$ will be determined later. 

Our goal is to prove that $\mathcal{Y}(t)$ approaches infinity in finite time.

Since $p>m$ and  $p \frac{m+1}{m} <6$, it follows that $m<5$. 
Therefore, $u$ and $w$ satisfy the regularity requirements on the test functions $\phi$ and $\psi$,
respectively, in Definition \ref{def:weaksln}. Replacing $\phi$
by $u$ in \eqref{wkslnwave}, $\psi$ by $w$ in \eqref{wkslnplt}, and using (\ref{def-Q}), we obtain
\begin{align}\label{6-10}
&Q(t) =(u_0,u_1)_\O+(\gamma
u_0+w_1,w_0)_\Gamma + \int^t_0(\|u_t\|^2_2+|w_t|^2_2)ds \nonumber\\
&\quad-\int^t_0(\|\nabla u\|^2_2+|\Delta
w|^2_2)ds + 2\int^t_0(\gamma u, w_t)_\Gamma ds - \int^t_0\int_\Omega |u_t|^{m-1}u_t u dx ds \nonumber\\
&\quad-\int^t_0\int_\Gamma |w_t|^{r-1}w_twd\Gamma ds + \int^t_0\int_\Omega
|u|^{p+1}\log|u| dxds + \int^t_0\int_\Gamma |w|^{p+1}\log|w| d\Gamma ds.
\end{align}

It is straightforward to verify that $Q(t)$ is absolutely continuous and therefore differentiable a.e. on $[0,T)$. Thus,
\begin{align}\label{6-9}
Q'(t)&=\|u_t(t)\|^2_2+|w_t(t)|^2_2-(\|\nabla
u(t)\|^2_2+|\Delta w(t)|^2_2)
-\int_\Omega |u_t|^{m-1}u_t u dx     -\int_\Gamma |w_t|^{r-1} w_t w d\Gamma      \nonumber\\
&\quad+2\int_\Gamma\gamma u \cdot w_t d\Gamma +\int_\Omega
|u|^{p+1} \log|u| dx+\int_\Gamma
|w|^{p+1} \log|w| d\Gamma,\ \ \mbox{a.e.}\
\mbox{on}\ [0,T).
\end{align}
By (\ref{def-Y}) and (\ref{6-9}), it follows that
\begin{align}\label{6-8}
\mathcal{Y}'(t)=(1-a)G^{-a}(t)G'(t)+\varepsilon Q'(t).
\end{align}

Using (\ref{B1}), (\ref{B3}), (\ref{6-9}), and (\ref{6-8}), we deduce that
\begin{align}\label{B5}
\mathcal Y'(t)&=(1-a)G^{-a}(t)(\|u_t\|^{m+1}_{m+1}+|w_t|^{r+1}_{r+1}) + \frac{p+3}{2} \varepsilon(\|u_t\|^2_2+|w_t|^2_2)\notag\\
&\quad +    \frac{p-1}{2} \varepsilon(\|\nabla u\|^2_2+|\Delta w|^2_2)+\frac{1}{p+1}\varepsilon(\|u\|^{p+1}_{p+1}+|w|^{p+1}_{p+1}) + (p+1) \varepsilon G(t)\notag\\
&\quad-\varepsilon\int_\O|u_t|^{m-1}u_tudx-\varepsilon\int_\G|w_t|^{r-1}w_twd\G+2\varepsilon\int_\G\gamma u\cdot w_td\G.
\end{align}

By H\"{o}lder's inequality, since $p>m$, we have
\begin{align}\label{B6}
\int_\Omega
|u_t|^{m-1} u_t u dx  \leq \|u_t\|^{m}_{m+1}\|u\|_{m+1} \leq C \|u_t\|^{m}_{m+1}\|u\|_{p+1}.
\end{align}
Combining inequality \eqref{B2} with identity \eqref{B3}, we obtain
\begin{align} \label{B6-1}
\|u\|^{p+1}_{p+1}+|w|^{p+1}_{p+1} \leq (p+1) S(t).
\end{align}
Together with \eqref{B6}, this yields
\begin{align}\label{B7}
\int_\Omega
|u_t|^{m-1} u_t u dx \leq C   S^{\frac{1}{p+1}}(t)\|u_t\|^{m}_{m+1} = C S^{\frac{1}{p+1}-\frac{1}{m+1}}(t)S^{\frac{1}{m+1}}(t)\|u_t\|^{m}_{m+1}.
\end{align}

From \eqref{B4} and \eqref{B7}, we have for any $\delta_1>0$,
\begin{align}\label{6-18}
\int_\Omega |u_t|^{m-1} u_t u dx & \leq C G^{\frac{1}{p+1}-\frac{1}{m+1}}(t)S^{\frac{1}{m+1}}(t)\|u_t\|^{m}_{m+1}\nonumber\\
&\leq G^{\frac{1}{p+1}-\frac{1}{m+1}}(t)\Big[\delta_1S(t)+C_{\delta_1}   \|u_t\|^{m+1}_{m+1}\Big]\nonumber\\
&\leq \delta_1G^{\frac{1}{p+1}-\frac{1}{m+1}}(t)S(t)+C_{\delta_1}      \|u_t\|^{m+1}_{m+1} G^{-a}(t)G^{a+\frac{1}{p+1}-\frac{1}{m+1}}(t)\nonumber\\
&\leq \delta_1G^{\frac{1}{p+1}-\frac{1}{m+1}}(0)S(t)+C_{\delta_1}    \|u_t\|^{m+1}_{m+1} G^{-a}(t)G^{a+\frac{1}{p+1}-\frac{1}{m+1}}(0).
\end{align}
Note that the positive constant $a$ satisfies (\ref{6-4}), and hence $a+\frac{1}{p+1}-\frac{1}{m+1}<0$. By the same argument, we obtain that for any $\delta_2>0$,
\begin{align}\label{6-19}
\int_{\Gamma} |w_t|^{r-1} w_t   w d\Gamma \leq
\delta_2 G^{\frac{1}{p+1}-\frac{1}{r+1}}(0)S(t) + C_{\delta_2}  |w_t|_{r+1}^{r+1}  G^{-a}(t)G^{a+\frac{1}{p+1}-\frac{1}{r+1}}(0).
\end{align}

Using (\ref{B3}), we have
\begin{align}\label{B8}
\|\nabla u\|^2_2 \leq \frac{2}{p+1}S(t).
\end{align}
By the trace embedding and \eqref{B8}, we obtain
\begin{align}\label{B9}
2\int_\G\gamma u\cdot w_td\G\leq C_*\|\nabla u\|_2|w_t|_2\leq C_*\left(\frac{2}{p+1}\right)^{\frac{1}{2}}S^{\frac{1}{r+1} - \frac{1}{2}}(t)S^{1-\frac{1}{r+1}}(t)|w_t|_2,
\end{align}
where  $C_*>0$ is the trace embedding constant satisfying $|\gamma u|_2\leq C_*\|\nabla u\|_2$. 
Since $r>1$, we have $\frac{1}{r+1} - \frac{1}{2} < 0$. Consequently, by (\ref{B4}), we have
$$
S^{\frac{1}{r+1} - \frac{1}{2}}(t)\leq  (p+1)^{\frac{1}{r+1} - \frac{1}{2}}     G^{\frac{1}{r+1} - \frac{1}{2}}(t).
$$
Together with \eqref{B9}, this yields that for any $\delta_3>0$,
\begin{align}\label{B10}
&2\int_\G\gamma u\cdot w_td\G \notag\\
&\leq C_*    \sqrt{2} (p+1)^{\frac{1}{r+1} -1}       
G^{\frac{1}{r+1} - \frac{1}{2}}(t) \left(\delta_3 S(t)+C_{\delta_3}|w_t|^{r+1}_{r+1}\right) \notag\\
&\leq C_* \sqrt{2} (p+1)^{\frac{1}{r+1} -1}  \delta_3     G^{\frac{1}{r+1} - \frac{1}{2}}(0)      S(t)
+C_{\delta_3} G^{\frac{1}{r+1} - \frac{1}{2}+a}(0) G^{-a}(t)|w_t|^{r+1}_{r+1},
\end{align}
since $G(t) \geq G(0)>0$, where $\frac{1}{r+1} - \frac{1}{2}+a <0$ by (\ref{6-4}).

Inserting \eqref{6-18}, \eqref{6-19}, and \eqref{B10} into \eqref{B5}, we obtain
\begin{align}\label{e3}
\mathcal{Y}'(t)&\geq\left[(1-a)-\varepsilon C_{\delta_1}G^{a+\frac{1}{p+1}-\frac{1}{m+1}}(0)\right]G^{-a}(t)\|u_t\|^{m+1}_{m+1}\notag\\
&\quad+\left[(1-a)-\varepsilon C_{\delta_2}G^{a+\frac{1}{p+1}-\frac{1}{r+1}}(0)
- \varepsilon C_{\delta_3}    G^{\frac{1}{r+1} - \frac{1}{2} +a}(0)\right]G^{-a}(t)|w_t|^{r+1}_{r+1}\notag\\
&\quad+ \frac{p+3}{2}\varepsilon(\|u_t\|^2_2+|w_t|^2_2) + \frac{p-1}{2}    \varepsilon (\|\nabla u\|^2_2+|\Delta w|^2_2)\notag\\
&\quad-\varepsilon  \Big[\delta_1 G^{\frac{1}{p+1}-\frac{1}{m+1}}(0)+\delta_2 G^{\frac{1}{p+1}-\frac{1}{r+1}}(0)
+C_*   \sqrt{2} (p+1)^{\frac{1}{r+1} -1}      \delta_3      G^{\frac{1}{r+1}- \frac{1}{2}}(0)\Big]S(t)\notag\\
&\quad +\frac{1}{p+1}\varepsilon\Big(\|u\|^{p+1}_{p+1}+|w|^{p+1}_{p+1}\Big) + (p+1) \varepsilon G(t).
\end{align}

By (\ref{B3}), we have
$$
\frac{1}{2}(\|u_t\|^2_2+|w_t|^2_2+\|\nabla u\|^2_2+|\Delta w|^2_2)=\frac{1}{p+1}S(t)-G(t)-\frac{1}{(p+1)^2}(\|u\|^{p+1}_{p+1}+|w|^{p+1}_{p+1}).
$$
Applying this equality to \eqref{e3}, we obtain
\begin{align}\label{B11}
&\mathcal{Y}'(t) \geq\left[(1-a)-\varepsilon C_{\delta_1}G^{a+\frac{1}{p+1}-\frac{1}{m+1}}(0)\right]G^{-a}(t) \|u_t\|^{m+1}_{m+1}\notag\\
&\quad+\left[(1-a)-\varepsilon C_{\delta_2}G^{a+\frac{1}{p+1}-\frac{1}{r+1}}(0) - \varepsilon C_{\delta_3}
G^{\frac{1}{r+1} - \frac{1}{2}+a}(0)\right]G^{-a}(t)|w_t|^{r+1}_{r+1} \notag\\
&\quad +   \frac{p-1}{4}   \varepsilon(\|u_t\|^2_2+|w_t|^2_2+\|\nabla u\|^2_2+|\Delta w|^2_2)\notag\\
&\quad+\varepsilon \Big[       \frac{p-1}{2(p+1)}    -  \delta_1 G^{\frac{1}{p+1}-\frac{1}{m+1}}(0)-\delta_2 G^{\frac{1}{p+1}-\frac{1}{r+1}}(0)
- C_* \sqrt{2} (p+1)^{\frac{1}{r+1} -1} \delta_3 G^{\frac{1}{r+1} -\frac{1}{2}}(0)\Big] S(t)\notag\\
&\quad+\left[\frac{1}{p+1}-   \frac{p-1}{2(p+1)^2}    \right]\varepsilon\Big(\|u\|^{p+1}_{p+1}+|w|^{p+1}_{p+1}\Big)+\left[(p+1)- \frac{p-1}{2}\right]\varepsilon G(t).
\end{align}

Substituting into \eqref{B11} the choices
\begin{align*}
&\delta_1=\frac{p-1}{12 (p+1)}G^{\frac{1}{m+1}-\frac{1}{p+1}}(0),\ \ \delta_2=\frac{p-1}{12(p+1)}G^{\frac{1}{r+1}-\frac{1}{p+1}}(0), \\
&\delta_3=\frac{p-1}{12 \sqrt{2}  C_*} (p+1)^{-\frac{1}{r+1} } G^{\frac{1}{2} - \frac{1}{r+1}}(0),
\end{align*}
it follows that
\begin{align}\label{6-20}
\mathcal{Y}'(t) &\geq \left[(1-a)-\varepsilon C_{\delta_1}G^{a+\frac{1}{p+1}-\frac{1}{m+1}}(0)\right]G^{-a}(t)\|u_t\|^{m+1}_{m+1}\notag\\
&\quad+\left[(1-a)-\varepsilon C_{\delta_2}G^{a+\frac{1}{p+1}-\frac{1}{r+1}}(0)
- \varepsilon C_{\delta_3} G^{\frac{1}{r+1} - \frac{1}{2} + a}(0)\right] G^{-a}(t)|w_t|^{r+1}_{r+1}\notag\\
&\quad+\frac{p-1}{4}\varepsilon \left(\|u_t\|^2_2+|w_t|^2_2+\|\nabla u\|^2_2+|\Delta w|^2_2 \right) + \frac{p-1}{4(p+1)}\varepsilon S(t)\notag\\
&\quad+ \frac{p+3}{2(p+1)^2} \varepsilon \left(\|u\|^{p+1}_{p+1}+|w|^{p+1}_{p+1} \right)+\frac{p+3}{2}\varepsilon G(t).
\end{align}

Since $0<a<\frac{1}{2}$ by \eqref{6-4}, for fixed $\delta_1,\delta_2,\delta_3>0$, we choose $0<\varepsilon<1$ sufficiently small so that
$$
 (1-a)-\varepsilon C_{\delta_1}G^{a+\frac{1}{p+1}-\frac{1}{m+1}}(0) \geq 0,
$$
and
$$
(1-a)-\varepsilon C_{\delta_2}G^{a+\frac{1}{p+1}-\frac{1}{r+1}}(0)-\varepsilon C_{\delta_3}G^{\frac{1}{r+1} - \frac{1}{2} +a}(0)  \geq 0,
$$
 to obtain from \eqref{6-20}  that there exists a constant $C>0$ such that
\begin{align}\label{6-22}
\mathcal{Y}'(t)&\geq 
C\varepsilon(\|u_t\|^2_2+|w_t|^2_2+\|\nabla u\|^2_2+|\Delta w|^2_2+\|u\|^{p+1}_{p+1}+|w|^{p+1}_{p+1}+G(t)+S(t)) \geq 0.
\end{align}
This implies that $\mathcal{Y}(t)$ is monotone increasing on $[0,T)$. Then, by (\ref{def-Y}), we have
$$
\mathcal{Y}(t)=G^{1-a}(t)+\varepsilon Q(t) \geq  \mathcal Y(0) = G^{1-a}(0)+\varepsilon Q(0).
$$

If $Q(0)\geq0$, then we do not need any  further condition on
$\varepsilon$. However, if $Q(0)<0$, we further choose $\varepsilon$ such
that $0<\varepsilon\leq-\frac{G^{1-a}(0)}{2Q(0)}$. In both cases, we have
\begin{align}\label{6-23}
\mathcal{Y}(t)\geq \frac{1}{2}G^{1-a}(0)>0,\ \ \mbox{for}\ t\in[0,T).
\end{align}

By \eqref{6-4}, we have $\frac{2}{(1-2a)(p+1)}<1$. Let $\delta= 1- \frac{2}{(1-2a)(p+1)}>0$. 
We claim that 
\begin{align}\label{B12}
\mathcal{Y}'(t)\geq C {\varepsilon}^{1+\delta }  \mathcal{Y}^{\frac{1}{1-a}}(t),\ \ \forall\ t\in [0,T).
\end{align}
Here, the exponent $\frac{1}{1-a} \in (1,2)$ since $0<a<\frac{1}{2}$.
We now prove (\ref{B12}).

Indeed, if $Q(t)\leq 0$ for some $t\in[0,T)$, then for such a value
of $t$, we obtain
\begin{align}\label{6-25}
\mathcal{Y}^{\frac{1}{1-a}}(t)=[G^{1-a}(t)+\varepsilon Q(t)]^{\frac{1}{1-a}}\leq G(t).
\end{align}
Then, from \eqref{6-22} and \eqref{6-25}, we infer that
$$
\mathcal{Y}'(t)\geq C\varepsilon G(t) \geq C   \varepsilon^{1+\delta}     \mathcal Y^{\frac{1}{1-a}}(t).
$$
If $Q(t)> 0$ for some $t\in[0,T)$, since $\varepsilon \leq 1$, we have $\mathcal{Y}(t)=G^{1-a}(t)+\varepsilon Q(t) \leq G^{1-a}(t)+ Q(t)$,
and therefore
\begin{align}  \label{insert1}
\mathcal{Y}^{\frac{1}{1-a}}(t) \leq  C \left(G(t) + Q^{\frac{1}{1-a}}(t)   \right).
\end{align}
By the Cauchy-Schwarz inequality and the trace embedding, we have
\begin{align} 
Q^{\frac{1}{1-a}}(t) \leq C\left(\|u\|^{\frac{1}{1-a}}_2\|u_t\|^{\frac{1}{1-a}}_2+|w|^{\frac{1}{1-a}}_2|w_t|^{\frac{1}{1-a}}_2+|w|^{\frac{1}{1-a}}_2\|\nabla u\|^{\frac{1}{1-a}}_2\right).\notag
\end{align}
By Young's inequality, since $p>1$, we have
\begin{align}\label{B13}
Q^{\frac{1}{1-a}}(t) \leq C\left(\|u_t\|^2_2+|w_t|^2_2+\|\nabla u\|^2_2+\|u\|^{\frac{2}{1-2a}}_{p+1}+|w|^{\frac{2}{1-2a}}_{p+1}\right).
\end{align}

Using \eqref{B3} and \eqref{B4}, and recalling $\delta= 1- \frac{2}{(1-2a)(p+1)}>0$, we have
\begin{align}\label{B15}
\|u\|^{\frac{2}{1-2a}}_{p+1} &=\Big(\|u\|^{p+1}_{p+1}\Big)^{\frac{2}{(1-2a)(p+1)}}\leq C [S(t)]^{\frac{2}{(1-2a)(p+1)}}= C  S^{ -\delta}(t)  S(t)    \notag\\
&\leq  C  G^{-\delta}(0)   S(t)  \leq C \varepsilon^{-\delta} S(t),
\end{align}
where we choose $\varepsilon \leq  G(0)$. 
Similarly,
\begin{align}\label{B16}
|w|^{\frac{2}{1-2a}}_{p+1}\leq C \varepsilon^{-\delta} S(t).
\end{align}
Substituting \eqref{B15} and \eqref{B16} into \eqref{B13}, and using (\ref{B8}), we obtain
\begin{align}\label{B16-1}
Q^{\frac{1}{1-a}}(t)\leq C(\|u_t\|^2_2+|w_t|^2_2+ S(t) +\varepsilon^{-\delta} S(t)) \leq C  \varepsilon^{-\delta} \left(  \|u_t\|^2_2+|w_t|^2_2+ S(t) \right).
\end{align}
Therefore, using \eqref{6-22} and (\ref{insert1}), we obtain
\begin{align}
\mathcal{Y}'(t) &\geq  C\varepsilon(\|u_t\|^2_2+|w_t|^2_2+G(t)+S(t)) \notag\\
& \geq C \varepsilon (  G(t) + \varepsilon^{\delta} Q^{\frac{1}{1-a}}(t) )
\geq C \varepsilon^{1+\delta} (G(t) +  Q^{\frac{1}{1-a}}(t) ) \geq  C \varepsilon^{1+\delta}  \mathcal{Y}^{\frac{1}{1-a}}(t), 
\end{align}
for all  $t\in[0,T)$ such that $Q(t)>0$. Therefore, in both cases, \eqref{B12} holds.

Since $\frac{1}{1-a} \in (1,2)$, the differential inequality (\ref{B12}) shows that $\mathcal Y(t)$ blows up in finite time. Moreover, the maximal life span $T$ satisfies the upper bound
\begin{align} \label{Tbound}
T \leq C \varepsilon^{-(1+\delta)} \mathcal{Y}^{-\frac{a}{1-a}}(0) \leq C \varepsilon^{-(1+\delta)} G^{-a}(0),
\end{align}
where the last inequality follows from (\ref{6-23}).  Since $0<\varepsilon \leq G(0)$, it follows from (\ref{Tbound}) that smaller initial data may lead to a longer life span,
which is consistent with physical intuition.

By (\ref{6-33}), the quadratic energy approaches infinity at the blow-up time $T$, namely,
$\limsup_{t\rightarrow T^-} E(t) = +\infty$. 
Furthermore, we claim that
\begin{align} \label{6-34}
\limsup_{t\rightarrow T^-} (\|\nabla u(t)\|_2^2 + |\Delta w(t)|_2^2) = +\infty.
\end{align}
Indeed, by \eqref{B3} and the fact that $\limsup_{t\rightarrow T^-} E(t) = +\infty$, we obtain 
\begin{align} \label{6-36}
\limsup_{t\rightarrow T^-} S(t) = + \infty.
\end{align}
Using the estimate in \eqref{3-11-1-1}, we derive 
\begin{align*}
\|\nabla u(t)\|_2^{p+1+\sigma} + |\Delta w(t)|_2^{p+1+\sigma} \geq C S(t),
\end{align*}
which, together with (\ref{6-36}), yields (\ref{6-34}). This completes the proof.
\end{proof}

\vspace{0.1 in}

\subsection{Blow-Up of Solutions with Positive Initial Energy}\label{blowup2}
In this subsection, we prove Theorem \ref{thm6-2}. This result states that the weak solution of system \eqref{PDE} blows up in finite time when the source terms are stronger than the damping terms, the initial total energy \(\widehat E(0)\) is positive but sufficiently small, and the initial quadratic energy \(E(0)\) is sufficiently large.

\begin{proof}[Proof of Theorem \ref{thm6-2}]
 We define the life span \(T\) of a weak solution \((u(t),w(t))\) to be the supremum of all \(T^*>0\) such that \((u(t),w(t))\) is a solution to system \eqref{PDE} on \([0,T^*]\).

By the assumptions \(p\frac{m+1}{m}<6\) and \(p>m\), we have \(p<5\). Using the estimate in (\ref{3-11-1-1}) and the embedding constants defined in (\ref{3-12}), we obtain
\begin{align}\label{B18}
\int_\Omega |u|^{p+1}\log|u| dx\leq \frac{1}{e (5-p)}\|u\|^{6}_{6} \leq \frac{\alpha_1}{e (5-p)}\|\nabla u\|^{6}_2 \leq \frac{8\alpha_1}{e (5-p)} (E(t))^3,
\end{align}
and 
\begin{align}\label{B19}
\int_\Gamma |w|^{p+1} \log|w| d\Gamma  \leq \frac{1}{e (5-p) } |w|^{6}_{6}   \leq \frac{\alpha_2}{e (5-p)} |\Delta w|^{6}_2 \leq  \frac{8\alpha_2}{e (5-p)} (E(t))^3.
\end{align}
Combining \eqref{B18} and \eqref{B19}, we have, for $t\in [0,T)$,
\begin{align}\label{b8}
\widehat{E}(t)& \geq E(t)-\frac{1}{p+1} \left(\int_\Omega |u|^{p+1}\log|u| dx + \int_\Gamma |w|^{p+1} \log|w| d\Gamma \right)\nonumber\\
&\geq E(t) - \frac{8(\alpha_1 + \alpha_2)}{(p+1) e (5-p)}(E(t))^{3}.
\end{align}
Recall that the function $\Psi:\mathbb{R}^+\to\mathbb{R}$ is defined by $\Psi(z)=z - \frac{8(\alpha_1 + \alpha_2)}{(p+1) e (5-p)} z^{3}$. 
Then \eqref{b8} is equivalent to
\begin{align}\label{b9}
\widehat{E}(t)\geq \Psi(E(t)),\ \ \forall\ t\in [0,T).
\end{align}
The function \(\Psi(z)\) is continuously differentiable and concave, and it attains its maximum at
$z = z_0 = \frac{1}{2} \left[ \frac{(p+1) e (5-p)}{6(\alpha_1 + \alpha_2) } \right]^{\frac{1}{2}} >0$. Define
$\hat{d}:=\sup_{[0,\infty)}\Psi(z)=\Psi(z_0)$. Note that $\Psi(z)$ is decreasing for $z>z_0$. Since $0\leq\widehat{E}(0)<\hat{d}=\Psi(z_0)$, 
there exists a unique $z_1 >z_0>0$ such that $\Psi(z_1)=\widehat{E}(0)$.
Thus, we infer from \eqref{b9} that
\begin{align}\label{b11}
\hat{d}=\Psi(z_0)>\Psi(z_1)=\widehat{E}(0)\geq \widehat{E}(t)\geq \Psi(E(t)),\ \ \forall\ t\in [0,T).
\end{align}
Recall that $\Psi(z)$ is decreasing and continuous for $z>z_0$, and that $E(t)$ is also continuous. Since we assume $E(0)>z_0$, it follows from \eqref{b11} that
\begin{align}\label{b12}
E(t)\geq z_1>z_0,\ \ \forall\ t\in[0,T).
\end{align}

We define
\begin{align}\label{defGG}
\widehat G(t) := B-\widehat{E}(t),
\end{align}
where the constant $B$ was introduced in (\ref{defB}). By (\ref{3-15}), we have
\begin{align}\label{GGin}
\widehat G'(t) = - \widehat{E}'(t) =  \|u_t\|^{m+1}_{m+1}+|w_t|^{r+1}_{r+1}\geq 0. 
\end{align}
Hence, $\widehat G(t)$ is non-decreasing in $t$.
From the assumption $\widehat E(0)<B$, we have $\widehat G(0) := B-\widehat{E}(0)>0$. Consequently,
\begin{align} \label{GG0}
\widehat G(t) \geq  \widehat G(0) >0,  \;\;\text{for all} \;\; t\in [0,T).
\end{align}

Recall that the function $Q(t)$ is defined in (\ref{def-Q}). We consider the function
\begin{align}\label{b14}
Y(t):=\widehat G^{1-a}(t)+\varepsilon Q(t),
\end{align}
for some $a\in (0,\frac{1}{2})$ satisfying (\ref{6-4}) and some $\varepsilon>0$. We will show that $Y(t)$ approaches infinity in finite time by choosing $\varepsilon$ sufficiently small.

Note that 
\begin{align}\label{b15}
Y'(t)=(1-a)\widehat G^{-a}(t)\widehat G'(t)+\varepsilon Q'(t),
\end{align}
where $Q'(t)$ is given in (\ref{6-9}).

We multiply (\ref{B3}) by $\frac{p+3}{2}$ and substitute $G(t) = \widehat G(t) - B$. It follows that
\begin{align} \label{B20-1}
\frac{p+3}{2} (\widehat G(t) - B) + \frac{p+3}{2} E(t) + \frac{p+3}{2(p+1)^2} \left( \|u\|_{p+1}^{p+1} + |w|_{p+1}^{p+1} \right) - \frac{p+3}{2(p+1)} S(t) =0.
\end{align}
Using (\ref{6-9}), (\ref{GGin}), (\ref{b15}) and (\ref{B20-1}), we obtain
\begin{align}\label{B20}
Y'(t)&= (1-a)\widehat G^{-a}(t)\Big[\|u_t\|^{m+1}_{m+1}+|w_t|^{r+1}_{r+1}\Big]+2\varepsilon(\|u_t\|^2_2+|w_t|^2_2)+\frac{p+3}{2}\varepsilon\widehat G(t)
\notag\\
&\quad-\frac{p+3}{2}\varepsilon B+\frac{p-1}{2}\varepsilon E(t)+\frac{p+3}{2(p+1)^2}\varepsilon(\|u\|^{p+1}_{p+1}+|w|^{p+1}_{p+1})+\frac{p-1}{2(p+1)}\varepsilon S(t)\notag\\
&\quad-\varepsilon\int_\O|u_t|^{m-1}u_tudx-\varepsilon\int_\G|w_t|^{r-1}w_twd\G+2\varepsilon\int_\G\gamma u\cdot w_td\G.
\end{align}

Recall from (\ref{SS}) and (\ref{defGG}) that $G(t) = \widehat G(t) - B$. Therefore, using (\ref{B3}) and \eqref{GG0}, we obtain
\begin{align}\label{B21}
\|\nabla u\|^2_2\leq \frac{2}{p+1}S(t)+2B.
\end{align}
From (\ref{3-9}) and (\ref{3-8}), we have  
\begin{align} \label{Ehat}
\widehat E(t) = E(t) - \frac{1}{p+1} S(t) + \frac{1}{(p+1)^2} \left( \|u\|_{p+1}^{p+1} + |w|_{p+1}^{p+1} \right).
\end{align}
Then, using (\ref{Ehat}), \eqref{Bz} and \eqref{b12}, we obtain
\begin{align} \label{b27}
\widehat G(t) = B-\widehat{E}(t) \leq B-E(t)+ \frac{1}{p+1} S(t) < z_0- z_1 + \frac{1}{p+1}S(t) < \frac{1}{p+1}S(t),
\end{align}
for $t\in [0,T)$.

Using H\"{o}lder's inequality, the trace embedding, and \eqref{B21}, we obtain 
\begin{align}\label{B22}
2\int_\G\gamma u\cdot w_t d\G & \leq 2|\gamma u|_2 |w_t|_2  \leq C_1\|\nabla u\|_2|w_t|_{r+1}\leq C_1\left(\frac{2}{p+1}S(t)+2B\right)^{\frac{1}{2}}|w_t|_{r+1}\notag\\
&\leq C_1\sqrt{\frac{2}{p+1}}S^{\frac{1}{2}}(t)|w_t|_{r+1}+\sqrt{2} C_1 B^{\frac{1}{2}}|w_t|_{r+1},
\end{align}
where $C_1 = 2 C_* |\Gamma|^{\frac{r-1}{2(r+1)}}$. Here, $C_*>0$ is the embedding constant satisfying $|\gamma u|_2\leq C_*\|\nabla u\|_2$.

Now we estimate the two terms on the right-hand side of \eqref{B22}. 

Since $r>1$, we apply Young's inequality to the first term on the right-hand side of \eqref{B22}:
\begin{align}\label{B25}
 C_1 \sqrt{\frac{2}{p+1}}S^{\frac{1}{2}}(t)|w_t|_{r+1}
&=  C_1\sqrt{\frac{2}{p+1}} S^{\frac{1}{2}-\frac{r}{r+1}}(t)S^{\frac{r}{r+1}}(t)|w_t|_{r+1}  \notag\\
&\leq C_1\sqrt{\frac{2}{p+1}} S^{\frac{1}{2}-\frac{r}{r+1}}(t) (\delta_3 S(t) + C_{\delta_3} |w_t|_{r+1}^{r+1}   ) \notag\\
&\leq C\delta_3\widehat G^{\frac{1}{2}-\frac{r}{r+1}}(0)S(t)+C_{\delta_3}\widehat G^{\frac{1}{2}-\frac{r}{r+1}+a}(0)\widehat G^{-a}(t)|w_t|^{r+1}_{r+1},
\end{align}
for any $\delta_3>0$, where the last inequality follows from (\ref{b27}) and (\ref{GG0}).

To estimate the second term on the right-hand side of (\ref{B22}), we apply Young's inequality:
\begin{align}\label{B23}
\sqrt{2} C_1 B^{\frac{1}{2}}|w_t|_{r+1}&=\sqrt{2}C_1 B^{\frac{1}{2}}|w_t|_{r+1} (\widehat G^{-a}(t))^{\frac{1}{r+1}}(\widehat G^{-a}(t))^{-\frac{1}{r+1}}\notag\\
& \leq \frac{(\sqrt{2}C_1)^{r+1}}{r+1}\widehat G^{-a}(t)|w_t|^{r+1}_{r+1}+\frac{r}{r+1}B^{\frac{r+1}{2r}}(\widehat G^{-a}(t))^{-\frac{1}{r}}\notag\\
&\leq \frac{(\sqrt{2}C_1)^{r+1}}{r+1}\widehat G^{-a}(t)|w_t|^{r+1}_{r+1}+\frac{r}{r+1}B^{\frac{r+1}{2r}}\left(\frac{1}{p+1}S(t)\right)^{\frac{a}{r}},
\end{align}
where the last inequality follows from (\ref{b27}). Using Young's inequality again, we obtain
\begin{align}\label{B24}
\frac{r}{r+1}B^{\frac{r+1}{2r}}\left(\frac{1}{p+1}S(t)\right)^{\frac{a}{r}}  \leq \frac{p-1}{4(p+1)}S(t) + C_2 B^{\frac{r+1}{2(r-a)}},
\end{align}
where $C_2 = \frac{r-a}{r}\left(\frac{4a}{r(p-1)}\right)^{\frac{a}{r-a}}  \left(\frac{r}{r+1} \right)^{\frac{r}{r-a}}$.

Inserting \eqref{B24} into \eqref{B23} yields
\begin{align} \label{B24-2}
\sqrt{2} C_1 B^{\frac{1}{2}}|w_t|_{r+1} \leq  \frac{(\sqrt{2}C_1)^{r+1}}{r+1}\widehat G^{-a}(t)|w_t|^{r+1}_{r+1} +  \frac{p-1}{4(p+1)}S(t) + C_2 B^{\frac{r+1}{2(r-a)}}.
\end{align}

From (\ref{B22}), (\ref{B25}) and (\ref{B24-2}), it follows that
\begin{align}\label{B26}
2\int_\G\gamma u\cdot w_td\G 
\leq &C\delta_3\widehat G^{\frac{1}{2}-\frac{r}{r+1}}(0)S(t)+C_{\delta_3}\widehat G^{\frac{1}{2}-\frac{r}{r+1}+a}(0)\widehat G^{-a}(t)|w_t|^{r+1}_{r+1} \notag\\
& +\frac{(\sqrt{2}C_1)^{r+1}}{r+1}\widehat G^{-a}(t)|w_t|^{r+1}_{r+1} +  \frac{p-1}{4(p+1)}S(t) + C_2 B^{\frac{r+1}{2(r-a)}}.
\end{align}

Using (\ref{Ehat}), (\ref{Bz}) and (\ref{b12}), we obtain
\begin{align*} 
&\frac{1}{(p+1)^2} \left( \|u\|_{p+1}^{p+1} + |w|_{p+1}^{p+1} \right)
=  \widehat E(t) - E(t) + \frac{1}{p+1} S(t)  \notag\\
&< B - E(t) + \frac{1}{p+1} S(t)  < z_0 - z_1 + \frac{1}{p+1} S(t) < \frac{1}{p+1} S(t).
\end{align*}
That is,
\begin{align} \label{B26-1}
\|u\|_{p+1}^{p+1} + |w|_{p+1}^{p+1} < (p+1) S(t).
\end{align}
Using (\ref{GG0}), (\ref{b27}) and (\ref{B26-1}), and following the same estimates as in (\ref{B6})--\eqref{6-19}, we infer that for any $\delta_1>0$,
\begin{align}
\int_\Omega |u_t|^{m-1} u_t udx \leq \delta_1 \widehat G^{\frac{1}{p+1}-\frac{1}{m+1}}(0)S(t)
+C_{\delta_1} \|u_t\|^{m+1}_{m+1}      \widehat G^{-a}(t)      \widehat G^{a+\frac{1}{p+1}-\frac{1}{m+1}}(0)     ,   \label{b28} 
 \end{align}
 and for any $\delta_2>0$,
 \begin{align}
\int_\Gamma |w_t|^{r-1} w_t w d\Gamma\leq
\delta_2 \widehat G^{\frac{1}{p+1}-\frac{1}{r+1}}(0)S(t)+C_{\delta_2}   |w_t|^{r+1}_{r+1}       \widehat G^{-a}(t) \widehat G^{a+\frac{1}{p+1}-\frac{1}{r+1}}(0).   \label{b29}
\end{align}
Inserting \eqref{B26}, \eqref{b28} and \eqref{b29} into \eqref{B20}, we have
\begin{align}\label{b30}
&Y'(t)\geq\left[(1-a)-\varepsilon C_{\delta_1} \widehat G^{a+\frac{1}{p+1}-\frac{1}{m+1}}(0)\right]\widehat G^{-a}(t)\|u_t\|^{m+1}_{m+1}\notag\\
&\quad+\left[(1-a)-\varepsilon C_{\delta_2}\widehat G^{a+\frac{1}{p+1}-\frac{1}{r+1}}(0)-\varepsilon C_{\delta_3}\widehat G^{\frac{1}{2}-\frac{r}{r+1}+a}(0)
- \varepsilon \frac{(\sqrt{2}C_1)^{r+1}}{r+1} \right]\widehat G^{-a}(t)|w_t|^{r+1}_{r+1}\notag\\
&\quad+2\varepsilon(\|u_t\|^2_2+|w_t|^2_2) +   \varepsilon   \frac{p+3}{2}  \widehat G(t) + \varepsilon \left(\frac{p-1}{2} E(t)  -\frac{p+3}{2} B  -C_2 B^{\frac{r+1}{2(r-a)}} \right)   \notag\\
&\quad+\varepsilon \Big[\frac{p-1}{4(p+1)}-\delta_1 \widehat G^{\frac{1}{p+1}-\frac{1}{m+1}}(0)-\delta_2\widehat G^{\frac{1}{p+1}-\frac{1}{r+1}}(0)-C\delta_3\widehat G^{\frac{1}{2}-\frac{r}{r+1}}(0)\Big]S(t)\notag\\
&\quad +\frac{p+3}{2(p+1)^2}\varepsilon\Big(\|u\|^{p+1}_{p+1}+|w|^{p+1}_{p+1}\Big).
\end{align}

In \eqref{b30}, we take $\delta_1,\delta_2,\delta_3>0$ sufficiently small so that
\begin{align} \label{b30-1}
\delta_1 \widehat G^{\frac{1}{p+1}-\frac{1}{m+1}}(0)+\delta_2\widehat G^{\frac{1}{p+1}-\frac{1}{r+1}}(0)+C\delta_3\widehat G^{\frac{1}{2}-\frac{r}{r+1}}(0) \leq \frac{p-1}{8(p+1)}.
\end{align}
For   fixed   $\delta_1,\delta_2,\delta_3>0$, we choose $\varepsilon>0$ sufficiently small so that
\begin{align} \label{b30-2}
(1-a)-\varepsilon C_{\delta_1} \widehat G^{a+\frac{1}{p+1}-\frac{1}{m+1}}(0) \geq 0,
\end{align}
and
\begin{align} \label{b30-3}
(1-a)-\varepsilon C_{\delta_2}\widehat G^{a+\frac{1}{p+1}-\frac{1}{r+1}}(0)-\varepsilon C_{\delta_3}\widehat G^{\frac{1}{2}-\frac{r}{r+1}+a}(0)
- \varepsilon \frac{(\sqrt{2}C_1)^{r+1}}{r+1} \geq 0.
\end{align}
By \eqref{b12}, $E(t)>z_0$ for all $t\in [0,T)$. Therefore, by \eqref{defB}, we have
\begin{align} \label{b30-4}
\frac{p-1}{4} E(t)  -\frac{p+3}{2} B  -C_2 B^{\frac{r+1}{2(r-a)}} >  \frac{p-1}{4}z_0  -\frac{p+3}{2} B  -C_2 B^{\frac{r+1}{2(r-a)}} =0.
\end{align}
Thus, combining \eqref{b30}--(\ref{b30-4}), we conclude that there exists a constant $C>0$ such that
\begin{align*}
Y'(t)\geq C \varepsilon(\|u_t\|^2_2+|w_t|^2_2+\widehat G(t)+E(t)+S(t)+\|u\|^{p+1}_{p+1}+|w|^{p+1}_{p+1})>0,
\end{align*}
for all $t\in [0,T)$.  Then, by an argument similar to that used in the proof of Theorem \ref{thm6-1}, we obtain
\begin{align*}
Y'(t)\geq C Y^{\frac{1}{1-a}}(t),\ \ \mbox{for}\ t\in[0,T).
\end{align*}
Therefore, the weak solution blows up in finite time. This completes the proof of Theorem \ref{thm6-2}.
\end{proof}

\vspace{0.1 in}

Finally, we prove Corollary \ref{cor1}, which states that if the source terms are stronger than the damping terms, the initial data belong to the unstable region $W_2$,
and the initial total energy is sufficiently small, then weak solutions blow up in finite time.

\begin{proof}[Proof of Corollary \ref{cor1}]
It is sufficient to show that \(E(0)>z_0\) if \((u_0,w_0)\in W_2\). Indeed, by the definition of \(W_2\) and (\ref{B18})--(\ref{B19}), we have
\begin{align} \label{corr-1}
\|(u_0,w_0)\|_X^2
= \|\nabla u_0\|_2^2+|\Delta w_0|_2^2
&< \int_\Omega |u_0|^{p+1}\log |u_0| \, dx
 + \int_\Gamma |w_0|^{p+1} \log |w_0| \, d\Gamma  \notag\\
&\leq  \frac{\alpha_1}{e(5-p)}\|\nabla u_0\|_2^{6}
+ \frac{\alpha_2}{e(5-p)} |\Delta w_0|_2^{6} \notag\\
&\leq \frac{\alpha_1+\alpha_2}{e(5-p)} \|(u_0,w_0)\|_X^6.
\end{align}
By the assumptions \(p\frac{m+1}{m}<6\) and \(p>m\), we have \(p<5\). 

Since \((u_0,w_0)\neq (0,0)\), we can divide both sides of \eqref{corr-1} by \(\|(u_0,w_0)\|_X^2\) to deduce that
\begin{align*}
\|(u_0,w_0)\|_X^2
> \left[\frac{e(5-p)}{\alpha_1+\alpha_2}\right]^{1/2}.
\end{align*}
It follows that
\begin{align*}
E(0) \geq \frac{1}{2}\|(u_0,w_0)\|_X^2
> \frac{1}{2}\left[\frac{e(5-p)}{\alpha_1+\alpha_2}\right]^{1/2}
> \frac{1}{2}\left[\frac{(p+1)e(5-p)}{6(\alpha_1+\alpha_2)}\right]^{1/2}
= z_0.
\end{align*}
This completes the proof of Corollary \ref{cor1}.
\end{proof}

\vspace{0.1 in}

\def\cprime{$'$}

\end{document}